\documentclass[a4paper]{amsart}
\usepackage{ricky}
\usepackage{stmaryrd}
\usepackage{MnSymbol}
\usepackage{graphicx}
\usepackage[all]{xy}
\usepackage{tikz-cd}
\usepackage{upgreek}
\usepackage{array}
  
\begin{document}
\title[Hida theory without ordinary locus]{Hida theory over some unitary Shimura varieties without ordinary locus}

%\date{\today}

\author{Riccardo Brasca}
\email{\href{mailto:riccardo.brasca@imj-prg.fr}{riccardo.brasca@imj-prg.fr}}
\urladdr{\url{http://www.imj-prg.fr/~riccardo.brasca/}}
\address{Institut de Math\'ematiques de Jussieu-Paris Rive Gauche\\
Universit\'e Paris Diderot\\
Paris\\
France}
\thanks{RB was partially supported by the ANR PerCoLaTor: ANR-14-CE25-0002-01.}

\author{Giovanni Rosso}
\email{\href{mailto:giovanni.rosso@concordia.ca}{giovanni.rosso@concordia.ca}}
\urladdr{\url{https://sites.google.com/site/gvnros/}}
\address{Concordia University, Departments of Mathematics and Statistics,
Montr\'eal, Qu\'ebec, Canada}
\thanks{Most of this work has been carried out while GR was Herchel Smith Postdoctoral Fellow at Cambridge University and Supernumerary Fellow at Pembroke College.}

 %da controllare

 %da controllare

\begin{abstract}
We develop Hida theory for Shimura varieties of type~A without ordinary locus. In particular we show that the dimension of the space of ordinary forms is bounded independently of the weight and that there is a module of $\Lambda$-adic cuspidal ordinary forms which is of finite type over $\Lambda$, where $\Lambda$ is a twisted Iwasawa algebra.
\end{abstract}

\maketitle

\section*{Introduction}

In the historical paper \cite{HidaENS}, Hida firstly constructed $\m Z_p\llbracket T \rrbracket$-adic families of ordinary eigenforms. Since then a lot of work has been done in many directions, both relaxing the ordinary hypothesis and considering more general automorphic forms. The state-of-the-art result concerning families of ordinary forms is \cite{hida_control}, where the author shows that one can do Hida theory for all algebraic groups whose associated Shimura variety has a non-empty ordinary locus. Over the years Hida theory has found a lot of spectacular applications in number theory, ranging from proving cases of Mazur--Tate--Teitelbaum conjecture \cite{GreenbergStevens} to  modularity lifting theorems \cite{BLGGT}. It would be hence extremely useful to have Hida theory in the greatest generality as possible;  the aim of this paper is to complete Hida's construction in the case of Shimura varieties of type~A under the mild assumption that $p$ is unramified in the Shimura datum.

Let us be more precise. In \cite{hida_control} Hida axiomatised his approach to the study of ordinary families of $p$-adic modular forms. His machinery requires a certain set of conditions  to be verified, such as:
\begin{itemize}
 \item the rank of the ordinary part of the space of weight $\kappa$ modular forms is bounded independently of $\kappa$;
 \item the existence of a Hasse invariant whose non-vanishing locus is the ordinary locus and which commutes with the $\U_p$ operators;
 \item a space of $p$-adic modular forms which contains as a dense subset all the classical modular forms, of all weights and all $p^n$-levels.
\end{itemize}

Of course he had in mind the standard notion of $p$-adic forms, \emph{i.e.} forms defined over the ordinary locus, but in the axiomatic approach what one needs is just a huge space of $p$-adic modular forms where one can embed the classical forms. Usually one takes the space of functions on the Igusa tower, a space that parametrises trivializations of the relevant $p$-divisible group over the ordinary locus. Hida was able to show that all the assumptions of his theory are verified in this case and so he obtained the theory for general PEL type Shimura varieties with ordinary locus. In the paper at hand we want to generalise Hida theory to Shimura varieties of PEL type without ordinary locus using the same machinery.

The first step is to replace the ordinary locus with the more general $\mu$-ordinary locus, that is always dense in the (reduction modulo $p$ of the) Shimura variety. The first two conditions can be checked thanks to the work on (reduced) $\mu$-ordinary Hasse invariants \cite{WushiMH, Valentinhasse}. What we still need is an analogue of the Igusa tower that let us define the space of $p$-adic modular forms. Of course in general the universal $\mu$-ordinary abelian scheme is not ordinary, but we have a good understanding of (the Dieudonné module of) his $p$-divisible group: it comes with a natural filtration whose graded pieces are generalised Lubin--Tate groups determined by the PEL data (more precisely by the signature at infinity of the unitary group). If the ordinary locus is not empty it is then equal to the $\mu$-ordinary locus and this graded pieces are actually the connected-multiplicative and the étale part of the $p$-divisible group. To define the usual Igusa tower, one looks at the automorphisms of the graded pieces that moreover respect the polarisation, so in practice one takes the automorphisms of the connected component of the universal ordinary $p$-divisible group. We do a similar thing: we take for the Igusa tower the automorphisms of the graded pieces, but \emph{we ignore the polarisation} and {\it we discard the \'etale part} (as, being \'etale, its module of invariant differentials is zero). This choice has two main effects:
\begin{itemize}
 \item our Igusa tower can be extended to a fixed toroidal compactification of the Shimura variety. Note that this would not be possible if we had taken into account the polarisation, as the universal semiabelian scheme is not polarised. We view this possibility of extending the Igusa tower to the compactification as one of the main reasons to justify our choice;
 \item the Igusa tower seems to be `too big', in the sense that its Galois group has too many  $p$-adic characters (because we want to take these characters as the weights of our modular forms and we know the number of components a weight must have from the complex theory).
\end{itemize}
To fix this issue, our idea is to consider only characters that are {\it locally analytic} in a suitable sense. Roughly speaking, it means that the action of the character on the Lie algebra of $\m C^\ast_p$ must have the same signature as the action on the Lie algebra of the Dieudonné module of the relevant graded pieces.

Let us be more explicit: we fix an unramified extension $\mc O$ of $\m Z_p$ and a type $(d,\mathfrak{f})$ in the sense \cite{Moonen}. The  Dieudonné module $M$ of the $\mu$-ordinary Barsotti--Tate with $\mc O$-structure of type $(d,\mathfrak{f})$  is a free $\mc O \otimes_{\m Z_p} \mc O$-module of rank $h$ (where $h$ depends on the type). It comes with a sub-$\mc O$-module  $\mr{Fil}^1$ that (and here is the big difference from the classical cases of Hida theory) is \emph{not} a free $\mc O \otimes_{\m Z_p} \mc O$-module, unless $M$ is ordinary in the usual sense. 

Suppose moreover that $h=1$ (this is what we call a generalised Lubin--Tate).  Classically, a $p$-adic weight would be a character of $\mr{Hom}_{\mr{cont}}(\mc O^\ast, \m C^\ast_p)$. But, given the set of all $\sigma:\mc O \rightarrow \mc O_{\m C_p}$, the action of $\mc O$ on $\mr{Fil}^1$ is via a subset of these $\sigma$'s. Our $p$-adic weights are defined as continuous characters of $\mc O^\ast$ which are $\sigma$-analytic for the $\sigma$'s appearing in the action of $\mc O$ over $\mr{Fil}^1$. (See Subsection~\ref{subsec: weight space} for the precise definition.)

We can then show that in this way we get characters (and hence weights) with the expected number of components. Note that in this case our weight space $\mc W$ and its Iwasawa algebra $\Lambda$ will be a twisted version of the usual weight space and Iwasawa algebra.

Thanks to a variant of the Hodge--Tate map, we prove in Proposition \ref{prop:classicaltopadic} that classical modular forms can be interpreted as functions on the Igusa tower, see Section \ref{sec:HodgeTate} for more details. We then define $V^{\mr{cusp}}$, a subspace of the functions on the Igusa tower where classical cuspidal forms are dense, and we check that all the conditions we need hold, assuming that one has a reduced $\mu$-ordinary Hasse invariant. The main theorem, stating the existence of cuspidal Hida families and a control result, is the following:
\begin{teono}
Let $G$ be a unitary group satisfying Assumption \ref{ass: p-adic}.  We have constructed:
\begin{enumerate}
\item An  ordinary projector $e_G=e_G^2$  on $V^{\mr{cusp}}$ such that the Pontryagin dual of its ordinary part 
\begin{equation*}
V^{\mr{cusp},*,\mr{ord}}\colonequals \mr{Hom}_{\m Z_p}\left(V^{\mr{cusp},\mr{ord}},\m Q_p/\m Z_p\right)
\end{equation*}
(which is naturally a $\Lambda$-module) is finite free over $\Lambda^{\circ}$.
\item The $\Lambda$-module of cuspidal Hida families $\mc S \colonequals \mr{Hom}_{\Lambda}\left(V^{\mr{cusp},*,\mr{ord}},\Lambda\right)$, which is of finite type over $\Lambda$.
\item  Given an algebraic weight $\kappa$ for $\Lambda$, let $\mc P_{\kappa}$ be the corresponding prime ideal of $\Lambda$. Then
\begin{equation*}
   \mc S\otimes \Lambda/\mc P_{\kappa}\stackrel{\sim}{\longrightarrow}\varprojlim_m\varinjlim_l V^{\mr{cusp}, \mr{ord}}_{m,l}[\kappa],
\end{equation*}
and, if $\kappa$ is very regular, combining it with \eqref{eqn:classicaltopadic} gives
\begin{equation*}
   {\mr S_{\kappa}(\mc H, K)}^{\mr{ord}} \stackrel{\sim}{\longrightarrow} \left(\mc S \otimes\Lambda /\mc P_{\kappa}\right)[1/p].
\end{equation*}
Here the maps are equivariant under the action of the unramified Hecke algebra away from $Np$ and the $\m U_p$-operators.
\end{enumerate}
\end{teono}
Assumption \ref{ass: p-adic} is simply to ensure that $p$ `does not ramify' in the Shimura datum.

A lot of people have recently been interested in the study of families of $p$-adic modular forms whenever the ordinary locus is empty. In \cite{EischenMantovan} Eischen and Mantovan develop and study extensively a theory of $p$-adic modular forms in this setting, using a slightly different Igusa tower. They also prove many results that we have not covered, and among these we quote the problem of density of classical modular forms in the space of $p$-adic modular forms, the construction of differential operators in the spirit of Serre's $\theta$ operator, and the construction of explicit families of $p$-adic modular forms. In \cite{ValentinU21} Hernandez constructs an eigenvariety for $\mr U(2,1)$ over a quadratic imaginary fields where $p$ is inert and, using the results of \cite{Valentinfiltration}, his method is very likely to generalise to all the cases considered in this paper. 

We remark that there exist Shimura varieties of type~D and PEL-type. In many case the ordinary locus is not empty and these cases have been dealt with in \cite{hida_control} but there are a few cases when the ordinary locus is empty. The situation in this case is even simpler then the one in our paper, as the $\mu$-ordinary $p$-divisible group is the product of a multiplicative part, an étale part, and  a fixed generalised Lubin--Tate \cite[\S 3.2.3]{Moonen}. The same method of this paper easily generalises to this setting, but for sake of brevity and exposition we shall not treat it. 

Our approach to $\mu$-ordinary $p$-divisible groups and, in general, a lot of our constructions are very `linear algebra-style'. This is done on purpose as we are very confident that the methods of the paper can be generalised to all Shimura variety with a flat and surjective map to the stack of $G$-Zip, such as Hodge type Shimura varieties (in whose context one has the work of Goldring--Koskivirta \cite{GoldringK1} and Zhang \cite{ZhangGzip}). This will be the subject of a future paper.

The plan of the paper is the following. In Section \ref{sec:Shimuradatum} we fix some notations for Shimura varieties and Barsotti--Tate groups with $\mc O$-structure, in Section \ref{sec:Igusatower} we define an Igusa tower $\mr{Ig}$ and compare classical forms with functions on $\mr{Ig}$. In Section \ref{sec:Heckeoperators} we define two ordinary projectors: $e_{\mr{GL}}$,  of representation theoretic flavor and which will be used to prove the control theorem, and $e_G$, of geometric nature and that will serve to control the rank of ordinary parts. In Section \ref{sec:Hidatheory} we put everything together to construct the module of cuspidal Hida families. Finally, in the appendix, we construct reduced mod $p$ Hasse invariants, using a remark of Hernandez.
\subsection*{Acknowledgments} It will be evident to the reader how much we owe to Haruzo Hida and his mathematics. We thank Vincent Pilloni for suggesting to one of us the problem and the idea for the Igusa tower and Valentin Hernandez for explaining us his constructions and for his feedback on the present paper. We also thank, for useful discussions, Wushi Goldring, Marc-Hubert Nicole, and Jacques Tilouine and all the participants of the reading group on Hida theory held in Paris 13 in 2013-14. Finally, we would to thank Giacomo Graziani and Toby Gee for pointing an error in a previous definition of our Iwasawa algebra.

\section{\texorpdfstring{Shimura datum and $\mu$-ordinary locus}{Shimura datum and mu-ordinary locus}}\label{sec:Shimuradatum}
In this section we introduce the Shimura varieties we will work with. These will be Shimura varieties of PEL type, associated with certain unitary groups.

\subsection{PEL data} \label{subsec: PEL data} Let $F_0$ be a totally real number field and let $F$ be a totally imaginary quadratic extension of $F$ ({\it i.e.}\ a CM field). Let $d$ be $[F_0 : \Q]$ and let $\tau_1,\ldots,\tau_d$ be the various embeddings $F_0 \hookrightarrow \mathbb{R}$. We choose once and for all a CM-type for $F$, {\it i.e.}\ we choose $\sigma_1,\ldots, \sigma_d$ embeddings $F \hookrightarrow \C$ such that $\sigma_{i|F_0} = \tau_i$. In particular, $\Hom(F,\C)= \set{\sigma_i, \bar \sigma_i}_i$, where $\bar \cdot$ is the complex conjugation.

Let $\m V \neq 0$ be an $\mc O_F$-lattice equipped with a nondegenerate bilinear pairing
\[
\langle \cdot, \cdot \rangle \colon \m V \times \m V \to \Z.\\
\]
We assume that $\langle \cdot, \cdot \rangle$ is symplectic with respect to $\mc O_F$, in the sense that it is alternating and moreover
\[
\langle \gamma x,y \rangle = \langle x, \overline\gamma y \rangle
\]
for all $\gamma \in \mc O_F$ and $x,y \in \m V$. We set $V \colonequals \m V \otimes_{\Z} \Q$. Note that the existence of the pairing $\langle \cdot, \cdot \rangle$ forces $\dim_{\Q}(V)$ to be even.
\begin{rmk} \label{rmk: basic situation}
Usually, the above data always arise from an Hermitian form as follows (if $F_0 = \Q$ this is always the case). Recall that a Hermitian form
\[
\Psi(\cdot, \cdot) \colon \m V \times \m V \to \mc O_F,
\]
is an $\mc O_{F_0}$-bilinear form that is $\mc O_F$-linear in the first variable and such that
\[
\Psi(y, x) = \overline{\Psi(y, x)}.
\]
Suppose we are given such a $\Psi$ that is nondegenerate. To obtain a pairing as above, we fix a totally negative element $\alpha \in \mc O_F$, so $\alpha$ is an algebraic integer such that $F=F_0(\sqrt{\alpha})$. Then, the $\mc O_{F_0}$-bilinear pairing
\begin{gather*}
\langle \cdot, \cdot \rangle \colon \m V \times \m V \to \Z\\
(x,y) \mapsto \Tr_{F/\Q} (\alpha\Psi(x,y))
\end{gather*}
is nondegenerate and symplectic with respect to $\mc O_F$.
\end{rmk}
\begin{rmk} \label{rmk: more general}
More generally one can start with a simple algebra $B$ over $\Q$, with center $F$ and a positive involution $^\ast \colon B \to B$ such that $[F:F_0] = 2$, where $F_0$ is the subfield of $F$ fixed by $^\ast$. This the so called Case (A) of \cite{pel}. (One need to use certain Morita equivalences in this case.) The above situation corresponds to $B=F$ and $\cdot^\ast = \overline{ \cdot }$.
\end{rmk}
Our choice of $\sigma_i$ gives an isomorphism $F \otimes_{F_0,\tau_i} \mathbb{R} \cong \C$ and, thanks to the symplectic pairing $\langle \cdot, \cdot \rangle$, the $\mathbb{C}$-vector space $V \otimes_{F_0,\tau_i} \mathbb{R}$ inherits a skew Hermitian form. We denote with $(a_i,b_i)$ its signature. (If we are in the case of Remark~\ref{rmk: basic situation}, these are the signatures of the Hermitian form $\Psi$.) Note that $a_i+b_i$ does not depend on $i$ and it is equal to $n\colonequals \frac{\dim_{\Q}(V)}{2d}$. Up to renaming the embeddings of $F_0$ into $\mathbb{R}$ we may (and we actually do) assume that
\[
a_1 \leq a_2 \leq \cdots \leq a_d \leq b_d \leq \cdots \leq b_2 \leq b_1.
\]

Let $G$ be the algebraic group scheme over $\Z$ given by the symplectic similitudes of $\m V$, i.e\ for all ring $R$ we have, functorially on $R$,
\[
G(R) = \set{(g, \lambda) \in \GL(\m V \otimes_{\Z} R) \times R^\ast \mbox{ such that } \langle gx,gy \rangle = \lambda \langle x,y\rangle \mbox{ for all } x,y \in \m V}.
\]
Note that $\lambda$ is uniquely determined by $g$, so we will usually drop it from the notation. Also, we have a morphism
\begin{gather*}
\nu \colon G \to \mathbb{G}_{\mathrm{m}} \\
(g,\lambda) \mapsto \lambda.
\end{gather*}
We have that $G_{\Q}$ is a connected reductive algebraic group and by construction there is an isomorphism
\[
G_{\mathbb{R}} \cong \mathrm{G}\left(\prod_{i=1}^d \U(a_i,b_i)_{\mathbb{R}} \right),
\]
where $\U(a_i,b_i)_{\mathbb{R}}$ is the real unitary associated to the standard Hermitian form of signature $(a_i,b_i)$. Here the notation $\mathrm{G}(\cdot)$ means that the similitude factors for the various embeddings $F_0 \hookrightarrow \mathbb{R}$ match.

We now fix a \emph{polarisation} of $(\m V, \langle \cdot,\cdot \rangle)$, i.e.\ an $\mathbb{R}$-algebra homomorphism
\[
h \colon \C \to \End_{F \otimes_{\Q} \mathbb R}(V \otimes_{\Q} \mathbb{R})
\]
such that, for all $z \in \C$ and all $x, y \in V \otimes_{\Q} \mathbb{R}$, we have:
\begin{itemize}
 \item $\langle h(z)x,y, \rangle = \langle x, h(\overline{z}y)\rangle$;
 \item the pairing $(x,y) \mapsto \langle x,h(i)y \rangle$ is symmetric and positive definite.
\end{itemize}
This is the same as giving an homomorphism $\mathbb{S} \to G_{\mathbb{R}}$, where $\mathbb{S}$ is the Deligne torus, and it induces a complex structure on $V \otimes_{\Q} \mathbb{R}$. We then have a decomposition
\[
V \otimes_{\Q} \C \cong V_{\C,1} \oplus V_{\C,2},
\]
where $h(z)$ acts on $V_{\C,1}$ via multiplication by $z$ and via multiplication by $\bar z$ on $V_{\C,2}$. There is an isomorphism of $F \otimes_{\Q} \mathbb{R} \cong \C^{d}$-modules
\[
V_{\C,1} \cong \prod_{i=1}^d \C^{a_i} \oplus \prod_{i=1}^d \overline \C^{b_i},
\]
where $\C$ acts by multiplication on the first factor and by multiplication by the conjugate on the second factor. The reflex field $E$ is by definition the field of definition of the isomorphism class of the complex $F$-representation $V_{\C,1}$.

\subsection{\texorpdfstring{The $p$-adic setting}{The p-adic setting}} \label{subsec: p-adic sett} We let $p \neq 2$ be a fixed prime number. We fix once and for all embeddings $\overline \Q \hookrightarrow \C$ and $i_p \colon\overline \Q \hookrightarrow \overline \Q_p$. We denote with $\mc P$ the corresponding prime ideal of $\mc O_E$ above $p$ and we write $E_{\mc P}$ for the $\mc P$-adic completion of $E$.
\begin{ass} \label{ass: p-adic}
From now on we assume that:
\begin{itemize}
 \item $p$ is unramified in $F_0$ and we denote by $\pi_1, \ldots, \pi_k$ the primes of $\mc O_{F_0}$ above $p$;
 \item the restriction of $\langle \cdot, \cdot \rangle$ to $\m V_p$ gives a perfect pairing with values in $\Z_p$.
\end{itemize}
\end{ass}
We write $d_i$ for the residue degree of $\pi_i$, so $d_1 + \cdots + d_k = d = [F_0 : \Q]$ and the completion of $\mc O_{F_0}$ at $\pi_i$ is isomorphic to $\Z_{p^{d_i}}$. In particular we have
\[
\mc O_{F_0,p} = \prod_{i=1}^k \Z_{p^{d_i}}.
\]
Note that $\mc O_{F,p}$ depends on which primes above $p$ in $\mc O_{F_0}$ are inert in $\mc O_F$: if $\pi_i$ is split in $\mc O_F$ we will write $\pi_i = \pi_i^+ \pi_i^-$. Then we have
\[
\mc O_{F,p} = \prod_{\pi_i \text{ split}} \left ( \Z_{p^{d_i}} \times \Z_{p^{d_i}} \right ) \times \prod_{\pi_i \text{ inert}} \Z_{p^{2d_i}},
\]
where the decomposition corresponding to split primes is obtained accordingly to $\pi_i = \pi_i^+ \pi_i^-$ and we have fixed an identification $\Z_{p^{d_i}} = \Hom_{\Z_p}(\Z_{p^{d_i}}, \Z_p)$.
\subsection{Shimura varieties} \label{subsec: shim var} Let $\mc H \subset G(\mathbb{A}^p)$ be a compact open subgroup that we assume to be \emph{sufficiently small}, fixed from now on. (To be precise we need $\mc H$ to be \emph{neat} in the sense of \cite[Definition~1.4.1.8]{lan}.) We are interested in the functor
\[
X \colon \mbox{ locally noetherian } \mc O_{E_{\mc P}}-\mbox{schemes} \to \mathbf{set}
\]
that to $S$ associates the isomorphism classes of the following data:
\begin{enumerate}
 \item an abelian scheme $A/S$ (of dimension $nd=\frac{\dim_{\Q}(V)}{2}$);
 \item a polarisation $\lambda \colon A \to A^\vee$ of degree prime to $p$;
 \item an action $\iota \colon \mc O_F \to \End_S(A)$ on $A/S$ such that $\lambda \circ \iota(\bar x) = \iota(x)^\vee \circ \lambda$ and $\det_{\Lie(A)}$ equals $\det_{V_{\C,1}}$ as polynomials with coefficients in $\mc O_S$;
 \item a $\mc H$-level structure $\alpha$ in the sense of \cite[Definition~1.3.7.6]{lan}.
\end{enumerate}
We furthermore require the usual determinant condition of Kottwitz, see \cite[Definition~1.3.4.1]{lan}. In our particular case, this last condition can be formulated in a simple way: it is equivalent to assume that $\Lie(A)$ is isomorphic, Zariski locally on $S$, to the $\mc O_S \otimes_{\Z_p} \mc O_{F,p}$-module
\[
\mc O_{F,p} = \mc O_S \otimes_{\Z_p} \left( \prod_{\pi_i \text{ split}} \left( \Z_{p^{d_i}}^{a_i} \oplus \Z_{p^{d_i}}^{b_i} \right) \times \prod_{ \pi_i \text{ inert}} \Z_{p^{2d_i}}^{a_i+b_i} \right).
\]
It is known (see for example \cite[Theorem~1.4.1.11 and Corollary~1.4.1.12]{lan}) that the functor $X$ is representable by a quasi-projective scheme over $\Sp(\mc O_{E_{\mc P}})$, denoted again by $X$. Thanks to the unramifiedness assumption, $X$ is smooth over $\Sp(\mc O_{E_{\mc P}})$. We have that $\dim(X) = \sum_{i=1}^d a_i b_i$.
We now let $K$ be a finite extension of $\Q_p$, that we assume to be `sufficiently big' (the meaning of this will change during the paper without any further comment). We base change $X$ to $\mc O_K$, using the same notation.
\subsection{Shimura varieties of Iwahoric level} \label{subsec: Iw level} Let $\mc A$ be the universal abelian scheme over $X$. Using the action of $\mc O_F$ on the $p$-divisible group associated to $\mc A$, we have a decomposition
\[
\mc A[p^\infty] = \prod_{ \pi_i \text{ split}} \left ( \mc A[(\pi_i^+)^\infty] \oplus \mc A[(\pi_i^-)^\infty] \right ) \times  \prod_{\pi_i \text{ inert}} \mc A[\pi_i^\infty]  ,
\]
where $\mc A[(\pi_i^-)^\infty]$ is canonically identified with the Cartier dual of $\mc A[(\pi_i^+)^\infty]$. We have an action of $\Z_{p^{d_i}}$ on $\mc A[(\pi_i^+)^\infty]$ and on $\mc A[(\pi_i^-)^\infty]$, and are both Barsotti--Tate groups of height $d_i(a_i+b_i)= d_i \frac{\dim_{\Q}(V)}{2d}$. The first one has dimension $d_ia_i$ and the second one has dimension $d_ib_i$. If $\pi_i$ is inert we have an action of $\Z_{p^{2d_i}}$ on $\mc A[(\pi_i)^\infty]$, that is an autodual $p$-divisible group of dimension $d_i(a_i+b_i)$. 
\begin{defi} \label{defi: Iw}
Let $X_{\Iw} \to \Sp(\mc O_K)$ be the functor
\[
X_{\Iw} \colon \mbox{ locally noetherian } \mc O_{K}-\mbox{schemes} \to \mathbf{set}
\]
that to $S$ associates the isomorphism classes of the following data:
\begin{enumerate}
 \item a point $(A, \lambda, \iota, \alpha) \in X(S)$;
 \item for all $i$ such that $\pi_i$ is split, a filtration
\[
0 = H_{i,0} \subset H_{i,1} \subset \cdots \subset H_{i, a_i+b_i} = A[\pi_i^+],
\]
where each $H_{i,j}$ is a finite and flat subgroup of $A[\pi_i^+]$, stable under the action of $\Z_{p^{d_i}}$ and of height $d_ij$;
 \item for all $i$ such that $\pi_i$ is inert, a filtration
\[
0 = H_{i,0} \subset H_{i,1} \subset \cdots \subset H_{i, a_i+b_i} = A[\pi_i],
\]
where $H_{i,j}$ is a finite and flat subgroup of $A[\pi_i]$, stable under the action of $\Z_{p^{2d_i}}$ and of height $2d_i j$. 
\end{enumerate}

This functor is representable by a scheme, denoted again $X_{\Iw}$ and there is a proper morphism $X_{\Iw} \to X$.
\end{defi}
\begin{rmk} \label{rmk: orth}
Let $(A, \lambda, \iota, \alpha, (H_{i,\bullet})_i)$ be a point of $X_{\Iw}$ and let $i$ be such that $\pi_i$ is split. Taking the orthogonal with respect to the perfect pairing given by Cartier duality between $A[\pi_i^+]$ and $A[\pi_i^-]$, we have that $H_{i,\bullet}$ induces an analogous filtration on $A[\pi_i^-]$, so the choice of working with $A[\pi_i^+]$ is harmless.
\end{rmk}

\subsection{Compactifications}
We fix one and for all a smooth toroidal compactification $X^{\tor}$ of $X$. It comes with the universal semi-abelian scheme $\mc G \to X^{\tor}$ and admits a stratification $ \bigsqcup_W X^{\mr{tor}}_W$, where $W$ ranges among $\mc H$-equivalence classes of  co-torsion-free sub-modules of $\m V$. We denote by $X^*$ the a minimal compactification  and by $\pi$ the proper map $\pi: X^{\mr{tor}}\rightarrow X^*$.

We let $S$ be the $\mu$-ordinary locus consisting of points whose associated $p$-divisible group is $\mu$-ordinary in the sense of \cite{Moonen}. We write $S^*$ for the $\mu$-ordinary locus of $X^*$ and $S^{\mr{tor}}$ for the counter-image of $S^*$ inside $X^{\tor}$. We denote by $\mr{Ha}^{\mu}$ a Hasse invariant for $G$. It is a modular form modulo $p$ whose non-vanishing locus is the $\mu$-ordinary locus. We shall sketch in the next section the idea behind its construction.

For a variety $Y$ over $W(k)$, we denote by $Y_m\colonequals Y \times_{\mr{Spec}(W(k))} \mr{Spec}(W(k)/p^m)$ and by $Y_{\infty}$ the corresponding formal scheme.

We conclude with three examples of algebraic group $G$ (and hence of Shimura variety $X$) that in our opinion the reader should keep in mind.
\begin{esempio}\label{esempio1}
We suppose $F_0 =\m Q$ and $F$ a quadratic imaginary field where $p$ is inert and that $G$ at infinity is $\mr {GU}(1,2)$. 
\end{esempio}
\begin{esempio}\label{esempio2}
Generalising the previous example, we suppose $F_0$ is a totally real field of degree $d$ and $F$ a CM where $p$ is inert, and that $G$ at infinity is $\mr G(\U(1,2d)\times \U(2,2d-1)\times \ldots \times \U(d,d+1))$. 
\end{esempio}
\begin{esempio}\label{esempio3}
We suppose $F_0$ is a totally real field of degree $2$ where  $p$ is inert and $F$ a CM where $p$ splits, and that $G$ at infinity is $\mr G(\U(1,4)\times \U(2,3))$. 
\end{esempio}

\subsection{Dieudonn\'e theory} \label{subsec: dieud}
In this section we recall the theory of $\mu$-ordinary $p$-divisible groups following  \cite{Moonen,Bijamu}. 
We fix an unramified extension $\m Q_{p^f}/\Q_p$ of degree $f$. We write $\mc O = \m Z_{p^f}$ for its valuation ring, $D$ for its Galois group and $\mr{Fr}$ for the Frobenius of $D$. Fix also a finite field $k$ such that $W(k)$, that will be our base ring, contains $\mc O$. We denote by $I$ the set of embeddings $\mc O \rightarrow W(k)$.

For a fixed integer $h$ we define $M$ to be a free $W(k)$-module  of rank $fh$ with basis $(e_{\sigma,i})_{\sigma \in I,i=1,\ldots,h}$ and we also fix $\eps: I^h \rightarrow \set{0,1}$ (this is equivalent to the map $\mathfrak{f} :I \rightarrow \set{0,1,\dots,h}$ in \cite{Moonen}). We define a Frobenius operator $F$ and a Verschiebung $V$ on $M$
\begin{align*}
F e_{\sigma,i}=p^{\eps(\sigma,i)}e_{\mr{Fr}\sigma,i} \;\:\;\: Ve_{\sigma,i}=p^{1-\eps(\sigma,i)}e_{\mr{Fr}^{-1}\sigma,i}.
\end{align*}
Clearly $FV=VF=p$. For all $i$, we let $\mc O$ act on $e_{\sigma,i}$ via the $\sigma$-embedding. Hence $M$ is a free $\mc O \otimes_{\Z_p}W(k)$-module of rank $h$. Via Dieudonn\'e equivalence, this defines a Barsotti--Tate group with $\mc O$-action of ($\mc O$-)height $h$ that shall be denote by $\mr{BT}_{\eps}$.

Take $h=1$ and identify $I$ with $\set{1,\ldots,f}$; we will sometimes write $\eps$ by writing the the images of the elements $\set{1,\ldots,f}$. We shall write $\eps^{(i)}$ for the map $\eps^{(i)}$ such that $\eps^{(i)}(j)=1$ if  $j \geq i+1 $ and $0$ otherwise.
For example, we have:
\begin{gather*}
\mr{BT}_{\eps^{(0)}} = \mr{BT}_{(1,1,\ldots,1)} = \mu_{p^{\infty}} \otimes_{\Z_p} \mc O; \\
\mr{BT}_{\eps^{(f-1)}} = \mr{BT}_{(0,\ldots,0,1)} = \mc{LT}_{\mc O}[p^{\infty}]; \\
\mr{BT}_{\eps^{(f)}} = \mr{BT}_{(0,\ldots,0)} = K/\mc O \;(\mbox{or }\Q_p/\Z_p \otimes_{\Z_p} \mc O).
\end{gather*}
If $\mc D$ denotes Cartier duality for $p$-divisible groups, note that $\mc D(\mr{BT}_{\eps})=\mr{BT}_{1-\eps}$. 
In general, we shall denote by $\mc{LT}_i$ the BT with $\mc O$-structure $\mr{BT}_{\eps^{(i)}} $ and call it  generalised Lubin--Tate. Note that a generalised Lubin--Tate is rigid, in the sense that there is a unique deformation in characteristic zero \cite[Corollary 2.1.5]{Moonen}.

We write $\mr{Fil}^1(M)$ for the $\mc O$-module spanned by the $e_{\sigma,i}$ for which $F(e_{\sigma,i}) \equiv 0 \bmod p$. (Note that it is not a free $\mc O \otimes_{\Z_p} W(k)$-module!)

For simplicity, we suppose now that $p$ is inert in the totally real field $F_0$ and we write $\mc O$ for the $p$-adic completion of its ring of integers. We shall say that we are in case ({U}) if $p$ is inert also in $F$ and in case ({L}) if $p$ splits. Recall that we fixed a unitary group with signature at infinity $(a_i,b_i)$. We normalise the $a_i$'s and $b_i$'s so that $a_1 \leq a_2\leq \cdots \leq a_d \leq b_d \leq \ldots \leq b_1$; this induces an isomorphism of $I$ with $\set{1,\ldots,f}$ that it will be fixed for the rest of the section. We also write $\sigma_i$ for the embedding in $\mr{Hom}(\mc O_{F} \otimes \m Z_p, \C_p)$ corresponding to $a_i$ and $\bar \sigma_i$ for the one corresponding to $b_i$. We let $n=a_1+b_1=a_i+b_i$. For the rest of the subsection, let  $\mc A^{\mu}$ be the $\mu$-ordinary abelian variety over $U$, a fixed open subset of the ordinary locus {\it in characteristic} $p$.

In  case ({L}), let $p=\pi^{+}\pi^{-}$. We have $\m Q_{p^f}=F_{0,p}$, $f=d$ and, over $\overline{\m{F}}_p$,  \begin{align}\label{eqn:DecompL}
\mc A^{\mu}[p^{\infty}] \cong \mc A^{\mu}[{(\pi^{+})}^{\infty}] \times \mc A^{\mu}[{(\pi^{-})}^\infty] \cong  \prod_{i=0}^f  \mc{LT}_i^{a_{i+1}-a_i} \times \mc D\left(\prod_{i=0}^f \mc{LT}_i^{a_{i+1}-a_i} \right),
\end{align} 
after defining $a_0=0$ and $a_{f+1}=n$. We shall denote the corresponding BT by $\mr{BT}_G$ and its Dieudonn\'e module by $M_G$. We have $\mr{dim}(A[{(\pi^{+})}^{\infty}])=\sum_{i=1}^f a_i$ and $\mr{dim}(A[{(\pi^{-})}^{\infty}])=\sum_{i=1}^f b_i$ (which is the dimension of $\mr{Fil}^1(M_G)$).

 Similarly in case (U) we have $\m Q_{p^f}=F_p$, $2d=f$ and \begin{align}\label{eqn:DecompU}
\mc A^{\mu}[p^{\infty}] \cong \prod_{i=0}^f  \mc{LT}_{i}^{a_{i+1}-a_i},
\end{align}  
where we have defined $a_{j+d}\colonequals b_{d-j+1}$. (Note that $a_{j+(d-j)+1} - a_{j+1+(d-j-1)+1}=b_{j}-b_{j+1}=a_{j+1}-a_j$ if $2 \leq j \leq d-1$. This ensure that $\mc A[p^{\infty}]$ is self-dual under Cartier duality.) We shall again denote the corresponding BT by $\mr{BT}_G$ and its Dieudonn\'e module by $M_G$.  As before we have $\mr{dim}(\mc A^{\mu}[p^{\infty}])=dn$ (which is the dimension of $\mr{Fil}^1(M_G)$).
\begin{rmk}
These explicit descriptions tell us immediately the condition for the ordinary locus to be non-empty in case (U): $a_1= \ldots=a_f$. 
\end{rmk}
In both cases, we can identify $M_G$ with $\mr H^1_{\mr {dR}}(\mc A^{\mu})$ and also $\mr{Fil}^1(M_G)$ with $\mr{Fil}^1 \mr H^1_{\mr {dR}}(\mc A^{\mu})$. 

In the Appendix \ref{Appendix} we shall defined a refined Hasse invariant. 
In the case all the signatures are different, the invariants of the appendix coincide with the ones of \cite{WushiMH, Valentinhasse} and it is no harm for the reader to think about them when considering the Hasse invariant, as we need our reduced Hasse invariants only in the proof of Theorem \ref{thm:boundedness}. Their Hasse invariant is the product of the partial Hasse invariants $\mr{Ha}^{\mu,\sigma_i}$ defined as the determinant of  
\begin{align*}
\frac{\bigwedge^{a_i} \mr{Fr}^{f}}{p^{c_i}} : \bigwedge^{a_i}  \mr H^1_{\mr {dR}}(\mc A^{\mu})_{\sigma_i} / \mr{Fil}^1 (\bigwedge^{a_i}  \mr H^1_{\mr {dR}}(\mc A^{\mu})_{\sigma_i})\rightarrow \bigwedge^{a_i}  \mr H^1_{\mr {dR}}(\mc A^{\mu})_{\sigma_i}/\mr{Fil}^1(\bigwedge^{a_i}  \mr H^1_{\mr {dR}}(\mc A^{\mu})_{\sigma_i}),
\end{align*}
where $c_i$ is an integer defined in the Appendix. This also shows that their Hasse invariant $\mr{Ha}^{\mu}$ is a modular form of parallel weight $p^f-1$ and that a power of it lifts to characteristic zero. 

We quote from \cite[Proposition 2.1.9]{Moonen}.
\begin{prop}\label{Prop:MoonenFil}
Let $X$ be a deformation of $\mr{BT}_G$ to a local Artinian $W(k)$-algebra. Then, \'etale locally on the $\mu$-ordinary locus, there is a unique filtration in BT with $\mc O$-structure  $X_i$ of $X$ such that $X_i$ lifts the slope filtration of $\mr{BT}_G$ (which is defined only \'etale locally). Moreover, we have a filtration $\left\{ \mc A_i \right\}_{i=0}^{f+1}$ for $\mc A^{\mu}[p^{\infty}]$, where $\mc A_0  = 0$, $\mc A_{f+1} = \mc A^{\mu}[p^{\infty}]$, and  $\mc A_1 =\mc A[p^{\infty}]^{\mr{mult}}$  ($\cong {(\mu_{p^{\infty}}\otimes \mc O)}^{a_1}$ in case (U) and  $\cong {(\mu_{p^{\infty}}\otimes \mc O)}^{a_1} \times {(\mu_{p^{\infty}}\otimes \mc O)}^{b_1}$ in case (L)). 
\end{prop}

\section{Modular forms and the Igusa tower}\label{sec:Igusatower}
\subsection{Classical modular forms}\label{sec:classicalforms}
Consider the group $P=\prod_{i=1}^{d}\mr{GL}_{a_i}\times \mr{GL}_{b_i}$. Its complex points coincide with the automorphisms of ${(V_{\C,1} \oplus V_{\C,2})}^{\m C= \mr{Id}}$, where $\m C= \mr{Id}$ means holomorphic part. Consider the semi-abelian variety $\mc G$ over $X^{\mr{tor}}$ and let us denote by $e$ the unit section $e:X^{\mr{tor}} \rightarrow \mc G$. We have a locally free sheaf $\omega \colonequals e^\ast \Omega^1_{\mc G/X^{\mr{tor}}}$. 

Recall we had fixed a CM type $(\sigma_i,\bar \sigma_i)_{i=1}^d$ for $(F_0,F)$; over $F$ we have a decomposition 
\begin{align*}
\omega \cong \bigoplus_{i=1}^d \omega_{\sigma_i} \oplus \omega_{\bar \sigma_i} 
\end{align*}
and we define $\mc E = \mc E^+ \oplus \mc E^-$, where
\begin{align*}
 \mc E^+ \cong & \bigoplus_{i} \isom(\mc O^{a_i}_{X^{\tor}},\omega_{\sigma}),\\ 
 \mc E^- \cong & \bigoplus_{i} \isom(\mc O^{b_i}_{X^{\tor}},\omega_{\bar \sigma_i}).
\end{align*}
We have that $\xi \colon \mc E \to X^{\tor}$ is an (algebraic) $P$-torsor. 

Let $\kappa=(k_{\sigma_i, 1}, \ldots, k_{\sigma_i, {a_i}},k_{\bar \sigma_i, 1}, \ldots, k_{\bar \sigma_i, {b_i}})_{i}$ be an algebraic weight. We say that $\kappa$ is:
\begin{itemize}
 \item \emph{dominant} if $k_{\sigma_i, 1} \geq \ldots \geq k_{\sigma_i, {a_i}} \geq k_{\bar \sigma_i, 1}\geq  \ldots \geq k_{\bar \sigma_i, {b_i}}$ for all $i$;
 \item \emph{regular} if $k_{\sigma_i, 1} > \ldots > k_{\sigma_i, {a_i}} > k_{\bar \sigma_i, 1} > \ldots >k_{\bar \sigma_i, {b_i}}$ for all $i$;
 \item \emph{big enough} if $\kappa$ is dominant and  $k_{\bar \sigma_i, {b_i}}>C$ for all $i$, with $C$ fixed, depending only on the Shimura datum. We shall write $\kappa \gg 0$;
 \item \emph{very regular} if  $\kappa$ is big enough and  $\kappa(w\alpha_p w \alpha_p^{-1}) \in p \m Z$ for all $i$, where $\alpha_{p}=(\alpha_{p,i})_i$ and $\alpha_{p,i}$ is the diagonal matrix with $1,p,\ldots,p^{a_i-1},1,p,\ldots,p^{b_i-1}$ on the diagonal. Here $w$ ranges over the elements in the Weyl group of $\prod_i \mr{GL}_{a_i}\times \mr{GL}_{b_i}$.
  
\end{itemize}
The sheaf of weight $\kappa$ automorphic forms is then
\begin{align*}
\omega^k= {\xi_{\ast} \mc O_{\mc E}}^{( \prod_i N_{a_i} \times N_{b_i})}[\kappa],
\end{align*}
being $ \prod_i N_{a_i} \times N_{b_i}$ the algebraic group of upper unipotent matrices in $P$. 
Locally for the Zariski topology, this sheaf is isomorphic to $R[\kappa]$,  the algebraic representation of $P$ of highest weight $\kappa$.
\begin{defin}
For any $\mc O_K$-algebra $R$ we define the space of weight $k$ modular forms as
\begin{align*}
\M_k(\mc H,R) \colonequals \Homol^0({X^*}_{/R},\pi_\ast\omega^\kappa)=\Homol^0({X^{\tor}}_{/R},\omega^\kappa)
\end{align*}
and the space of cusp forms 
\begin{align*}
\mr S_k(\mc H,R) \colonequals \Homol^0({X^*}_{/R},\pi_\ast(\omega^\kappa(-D)))=\Homol^0({X^{\tor}}_{/R},\omega^\kappa(-D)),
\end{align*}
where $D$ denotes the boundary of $X^{\tor}$.
\end{defin}
\subsection{The Igusa tower}
Recall that $S$ is the $\mu$-ordinary locus of our Shimura variety. Thanks to Proposition \ref{Prop:MoonenFil} we can define over $S$
\begin{align*}
\mr{Gr}^{\bullet}(\mc A^{\mu} [p^{\infty}])= \prod_{i=1}^{f+1} \mc A_i/\mc A_{i-1}, \;\;\; \mr{Gr}^{\bullet}(\mc A^{\mu}[p^{\infty}]^{\circ})= \prod_{i=1}^{f} \mc A_i/\mc A_{i-1}.
\end{align*}
The interest in removing the \'etale part lies in the fact that the connected part of $\mc A^{\mu}$ can be extended naturally to the whole $S^{\mr{tor}}$. Indeed, let $W$ be a cusp label for a component of the boundary of the toroidal compactification of rank $r$ (hence $r \leq a_1$). Over $X^{\mr{tor}}_W$ we have the semi-abelian scheme 
\begin{align*}
0 \rightarrow \m G_m\otimes \mc O \otimes W \rightarrow \mc G \times_{X^{\mr{tor}}} X^{\mr{tor}}_W \rightarrow \mc A_W \rightarrow 0,
\end{align*}
where $\mc A_W$ is the universal abelian variety associated with the Shimura datum for $W$. If we restrict to the ordinary locus and take the $p^l$-torsion we obtain
\begin{align*}
0 \rightarrow {(\mu_{p^l} \otimes \mc O)}^r\rightarrow \mc G \times_{X^{\mr{tor}}} X^{\mr{tor}}_W[p^l]\rightarrow \mc A^{\mu}_W[p^l] \rightarrow 0.
\end{align*}
Using (\ref{eqn:DecompL}) or (\ref{eqn:DecompU}) we see that the difference between $A^{\mu}_W[p^l]$ and $A^{\mu}[p^l]$ is given by $r$ copies of $\mu_{p^l} \otimes \mc O$ and $\m Z/p^l \m Z \otimes \mc O$. Hence, the $\mr{Gr}^{\bullet}(\mc A^{\mu,\circ})$  and $\mr{Gr}^{\bullet}(\mc G^{\mu,\circ})$ are the same, or better the first extends to the second. 

Hence, we can define 
\begin{align*}
\mr{Ig}_{m,l} \colonequals  & \mr{Isom}_{\mc O_F \otimes \m Z_p -\mr{BT}, S^{\mr{tor}}_m}(\mr{Gr}^{\bullet}( \mc{G}_m^{\mu,\circ}[p^{l}]), \mr{BT}_{G,m}^{\circ}[p^l]),\\
\mr{Ig} \colonequals  & \varinjlim_m\varprojlim_l \mr{Ig}_{m,l}.
\end{align*} 
Note that we require the isomorphisms to respect the $+$ and $-$ parts in case (L) but do not require them to respect the polarisation condition (as $\mc G$ is no longer polarised).

 %Also, here one can see why Hernandez does not use the overconvergence of the slope filtration \cite[Proposition 5.27]{Valentinfiltration}.  
\begin{teo}
The formal scheme $\mr{Ig}$ is a pro-\'etale Galois cover of $S^{\mr{tor}}_{\infty}$ of Galois group $\prod_{i=0}^{f-1} \mr{GL}_{a_{i+1}-a_{i}}(\mc O)$ in case (U) and $\prod_{i=0}^{f-1} \mr{GL}_{a_{i+1}-a_{i}}(\mc O) \times \prod_{i=1}^{f}\mr{GL}_{b_{i}-b_{i+1}}(\mc O)$ in case (L).
\end{teo}
The Galois group that appears as automorphisms of the Igusa tower is the group of automorphisms of the (connected part of the) $\mu$-ordinary BT. We denote by $N_{\mr{Ig}}$ the  upper unipotent part of $\mr{Gal}(\mr{Ig},S^{\mr{tor}})$ and by  $T_{\mr{Ig}}$ the  corresponding torus. We have that $T_{\mr{Ig}}$  is isomorphic to ${(\mc O^\ast)}^{a_f}$ in case (U) and ${(\mc O^\ast)}^{a_f} \times {(\mc O^\ast)}^{b_f}$ in case (L).

We can also push-forward $\mr{Ig}$ to the minimal compactification and we shall denote it by $\mr{Ig}^*$. We shall write $\mc I^0$ and $\mc I^{*,0}$ for the sheaves of ideal of the boundary of the two Igusa tower. Note that $\pi_* \mc O_{\mr{Ig}}= \mc O_{\mr{Ig}^*}$ and $\pi_* \mc I^0 = \mc I^{*,0}$ for the same reasoning as \cite[Theorem 7.2.4.1]{lan}. (See also \cite[Theorem 6.2.2.1]{lan_ram}.)

The attentive reader will have noticed that this group has too many characters. Taking $G$ as in Example  \ref{esempio1} with $p$ inert in $F$, our Igusa tower has Galois group isomorphic to $\mc O^\ast \times \mc O^\ast$. The usual Hida theory would take for the weight space $\mr{Hom}_{\mr{cont},\m Z_p}(\mc O^\ast \times \mc O^\ast, \m C_p^\ast)$ while the classical weights are only three. Inspired by \cite{shimura}, we define the weight space $\mc W$ to be $\mr{Hom}_{\mr{cont}}(\mc O^\ast,\m C_p^\ast)\times \mr{Hom}_{\mr{cont},\mc O}(\mc O^\ast,\m C_p^\ast)$ where $\mr{Hom}_{\mr{cont},\mc O}$ means that we consider continuous characters such that the induced morphism between Lie algebras is $\mc O$-linear. The reason is that the $\mr{Fil}^1$ of the Dieudonn\'e module of $\mu_{p^{\infty}}$ is a two dimensional module (over $\mc O$) where $\mc O$ acts both via $\sigma: \mc O \rightarrow \mc O $ and $\bar \sigma: \mc O \rightarrow \mc O $, while on the  $\mr{Fil}^1$ of the Dieudonn\'e module of $\mc {LT}_{\mc O}$ the action of $\mc O$ is only via $\sigma$. Every generalised $ \mc {LT}_i$ will hence contribute to the weight space only according to the character of $\mathcal{O}$ acting on its $\mr{Fil}^1$. Following this idea we now introduce our weight space.
\subsection{The weight space} \label{subsec: weight space}
Our weight space will parametrise characters that are `partially locally $\mc O$-analytic', in the sense that we want to control the induced morphism between Lie algebras. Let us be more precise.

Let $A$ be a $K$-affinoid algebra and let $\chi \in \mr{Hom}_{\mr{cont}}(\mc O^\ast,A^\ast)$ be a continuous character. We have an induced morphism $d \chi \colon \Lie(\mc O^\ast) \to \Lie (A^\ast)$. This extends to a morphism, denoted with the same symbol,
\[
d \chi \colon \Lie(\mc O^\ast) \otimes_{\Z_p} \mc O \to \Lie (A^\ast).
\]
The Lie algebra of $\mc O^\ast$ and of $A^\ast$ can be naturally identified with $\mc O$ and $A$, respectively. Moreover we can identify $\Lie(\mc O^\ast) \otimes_{\Z_p} \mc O = \mc O \otimes_{\Z_p} \mc O$ with $\prod \mc O$, where the product is over the embeddings of $\mc O$ in $\m C_p$. In practice we get a morphism
\[
d \chi \colon \prod_{\mr{Hom}(\mc O,\m C_p)} \mc O \to A.
\]
\begin{defi} \label{defi: J analytic}
Let $J$ be a subset of $\mr{Hom}(\mc O,\m C_p)$ and let $f \colon \Lie(\mc O^\ast) \to \Lie(A^\ast)$ be a $\Z_p$-linear morphism. We say that $f$ is \emph{$J$-linear} if the induced morphism $\mc O \otimes_{\Z_p} \mc O \to A$ factors through a map
\[
\prod_{J} \mc O \to A.
\]
We will write $\Hom_J(\Lie(\mc O^\ast), \Lie(A^\ast))$ for the set of $J$-linear morphism. Following \cite[D\'efinition 4.1]{deieso}, we say that $\chi \in \mr{Hom}_{\mr{cont}}(\mc O^\ast,A^\ast)$ is \emph{locally $J$-analytic} if $d \chi$ is $J$-linear. Similarly, as in {\it loc. cit.}, we have the more general notion of a locally $J$-analytic function on $\mc O$.		
\end{defi}
Let $i$ be an integer between $0$ and $f$. We define $J(\mc{LT}_i)$ to be the set of $\sigma_j$ for $j \geq 1+i$ and $J(\mc D(\mc{LT}_i))$ to be the set of $\sigma_j$ for $j \leq i$. We also set
\[
\mr{Hom}_{\mr{cont},\mc{LT}_i}(\mc O^\ast,A^\ast) \colonequals \set{\chi \in \mr{Hom}_{\mr{cont}}(\mc O^\ast,A^\ast) \mbox{ such that } \chi \mbox{ is } J(\mc{LT}_i)-\mbox{analytic}}.
\]
The definition of $\mr{Hom}_{\mr{cont},\mc D(\mc{LT}_i)}(\mc O^\ast,A^\ast)$ is analogous.
We call $\mc W_{\mc{LT}_i}$ the functor
\[
A \mapsto \mr{Hom}_{\mr{cont},\mc{LT}_i}(\mc O^\ast,A^\ast).
\]
By definition we have that $\mc W_{\mc{LT}_i}(A) \subseteq \mr{Hom}_{\mr{cont}}(\mc O^\ast,A^\ast)$ and $\mr{Hom}_{\mr{cont}}(\mc O^\ast,\cdot)$ is representable by the disjoint union of $p^f-1$ open unit disks of dimension $f$ (via $\mc O^\ast \cong \mu_{p^f-1} \times \mc O \cong \mu_{p^f-1} \times \Z_p^f$). Letting $\mathbf{D}_f$ be such a disk, by definition the diagram
\[
\begin{tikzcd}
\mc W_{\mc{LT}_i}(A) \arrow[hook]{r} \arrow{d}{d} & \coprod \mathbf{D}_f(A) \arrow{d}{d} \\
\mr{Hom}_{J(\mc{LT}_i)}(\Lie(\mc O^\ast), \Lie(A^\ast)) \arrow[hook]{r} & \mr{Hom}(\Lie(\mc O^\ast), \Lie(A^\ast))
\end{tikzcd}
\]
is cartesian. We can follow word by word the approach of \cite[Sections 1 and 2]{fourier} to show that $\mc W_{\mc{LT}_i}$ is representable by a closed analytic subvariety, denoted again $\mc W_{\mc{LT}_i}$, of $\coprod \mathbf{D}_f$. Similarly, we have the closed subvariety $\mc W_{\mc D(\mc{LT}_i)}$
\begin{prop} \label{prop: descr W}
We have that $\mc W_{\mc{LT}_0} = \coprod \mathbf{D}_f(A)$. In general, the base change to $\m C_p$ of $\mc W_{\mc{LT}_i}$ is isomorphic to the disjoint union of $p^f-1$ copies of an $(f-i)$-dimensional open unit disk. Similarly, $\mc W_{\mc D(\mc{LT}_f)} = \coprod \mathbf{D}_f(A)$ and the base change to $\m C_p$ of $\mc W_{\mc D(\mc{LT}_i)}$ is isomorphic to the disjoint union of $p^f-1$ copies of an $i$-dimensional open unit disk.
\end{prop}
\begin{proof}
The statement for $\mc W_{\mc{LT}_0}$ is obvious. For the general case we consider the Lubin-Tate group
\[
\mc G \colonequals \prod_{j \geq i + 1} \BT_{(0,\ldots,0,1,0,\ldots,0)},
\]
where the $1$ is in the $j$-position. (Note that this is \emph{not} $\mc{LT}_i$ but the first filtered piece of the Dieudonn\'e module of the two are isomorphic as $\mc O \otimes_{\m Z_p} \mc O$-modules.) Using $\mathcal G$,  the proof is exactly the same as in \cite[Section 3]{fourier} (using $\mc O^\ast \cong \mu_{p^f-1} \times \mc O$). The proof for $\mc W_{\mc D(\mc{LT}_i)}$ is similar.
\end{proof}
\begin{rmk} \label{rmk: not is fin ext}
If $i= f- 1$ then, by \cite[Section 3]{fourier}, we know that $\mc W_{\mc{LT}_i}$ is \emph{not} isomorphic to a open unit disk, even after base change to any finite extension of $K$. We conjecture that the same hold for any $i \neq 0, f$.
\end{rmk}
We now move on to define our Iwasawa algebra.
\begin{defi} \label{defi: Lambda}
Let us denote by $\mc W^{\circ}_{\mc{LT}_i}$ the connected component of $\mc W_{\mc{LT}_i}$ containing the trivial character. We set
\[
\Lambda^\circ_{\mc{LT}_i} \colonequals \mc O_{\mc W^\circ_{\mc{LT}_i}}(\mc W^\circ_{\mc{LT}_i})^{\leq 1} \mbox{ and } \Lambda_{\mc{LT}_i} \colonequals \mc O_{\mc W_{\mc{LT}_i}}(\mc W_{\mc{LT}_i})^{\leq 1},
\]
where $\mc O(\cdot)^{\leq 1}$ means the ring of analytic functions bounded by $1$. The definition of $\Lambda^\circ_{\mc{D}(\mc{LT}_i)}$ and $\Lambda_{\mc{D}(\mc{LT}_i)}$ is the same.
\end{defi}
\begin{remark} \label{rmk: Lambda complicated}
The ring-theoretic properties of $\Lambda_{\mc{LT}_i}$ are quite mysterious, for example it does not satisfy Weierstra\ss{} preparation theorem (see the proof of \cite[Lemma~3.9]{fourier}, where the authors construct an element of $\Lambda_{\mc{LT}_{f-1}}$ that has infinitely many zeros) and we do not know if it is noetherian. Moreover, we ignore whether it is an admissible ring, so we can not take the generic fiber of $\Spf(\Lambda_{\mc{LT}_i})$.
\end{remark}
Given a point $x \in \mc W_{\mc{LT}_i}$, we denote by $P_x$ the prime ideal of $\Lambda_{\mc{LT}_i}$ of functions vanishing at $x$ and similarly for $\Lambda_{\mc{D}(\mc{LT}_i)}$.
\begin{prop} \label{prop: properties of Lambda}
We have the following:
\begin{enumerate}
 \item The ring $\mc O_{\mc W^\circ_{\mc{LT}_i}}(\mc W^\circ_{\mc{LT}_i})$ is isomorphic to the continuous dual of the locally convex vector space of locally $J(\mc{LT}_i)$-analytic functions $C^{\mc{LT}_i-\mathrm{an}}(\mc O, K)$.
 \item The natural restriction morphism
\[
\mc O_K \llbracket X_1, \ldots, X_f \rrbracket = \mc O_{\mathbf{D}_f}(\mathbf{D}_f)^{\leq 1} \hookrightarrow \Lambda_{\mc{LT}_i}^\circ
\]
is injective.
\item The set of prime ideal $P_x$ for the points $x: z \mapsto \prod_{\sigma \in J(\mc{LT}_i)} \sigma(z)^{k_{\sigma}}$ with $k_{\sigma} \gg 0$ is Zariski dense. 
\end{enumerate}
The same is true for $\mc{D}(\mc{LT}_i)$ instead of $\mc{LT}_i$.
\end{prop}
\begin{proof}
\begin{enumerate}
 \item One can find explicit equations for $\mc W^{\circ}_{\mc{LT}_i}$ as in \cite[Corollary 1.5]{fourier} or \cite[Remark 1.14]{BSX}; the statement is then a consequence of Amice's transform as in \cite[Theorem 2.3]{fourier}
 \item The proof for bounded functions is exactly the same as \cite[Lemma 1.15]{BSX} using $C^{\mc{LT}_i-\mathrm{an}}(\mc O, K)$ in place of $C^{\mathrm{an}}(\mc O, K)$. Then one notices that $1 \geq |f|_{{\mathbf{D}_f}} \geq |f|_{\mc W^{\circ}_{\mc{LT}_i}}$ for all $f \in \mc O_{\mc W^\circ_{\mc{LT}_i}}(\mc W^\circ_{\mc{LT}_i})$.
 \item After the (flat) extension to $\mc O_{\C_p}$, we have the isomorphism $\Lambda^{\circ}_{\mc{LT}_i} \ctens_{\mc O_K} \mc O_{\C_p} \cong \mc O_{\C_p} \llbracket T_1, \ldots, T_{f-i}\rrbracket$. Let $f$ in $\Lambda_{\mc{LT}_i}^\circ$ be an element which vanishes on all $P_x$ as in the statement of the proposition. Viewing $f$ as an element of $\mc O_{\C_p} \llbracket T_1, \ldots, T_{f-i}\rrbracket$, the evaluation at a point $x$ corresponds to $T_i \mapsto t_{i,x}$ for some $t_{i,x} \in \C_p$. By \cite[Proposition 1.20]{BSX}, the norm of $(t_{i,x})_i$ as a point of $\mathbf{D}_{f-i, \C_p} = \Spf(\mc O_{\C_p} \llbracket T_1, \ldots, T_{f-i}\rrbracket)^{\rig}$ is determined by the norm of the corresponding point $x=(x_1, \ldots x_f) \in \mathbf{D}_{f, \C_p}$. In particular, infinitely many of the points $x \in \mathbf{D}_{f-i, \C_p}$ considered in the proposition are in a fixed closed ball of radius strictly smaller than $1$. Iterating Strassman's Theorem, \cite[Theorem, \S 6.2.1]{Robert}, we see that an element of $\mc O_{\C_p} \llbracket T_1, \ldots, T_{f-i}\rrbracket$ vanishing on all these points is $0$.
 % It is enough to show that if an element of $\mc O_{\mc W^\circ_{\mc{LT}_i}}(\mc W^\circ_{\mc{LT}_i})$ is zero on all such $x$, then it is $0$. But polynomial functions are dense among continuous functions \cite[Corollary 1.3]{UdiBordeaux}, so in particular they are dense in $C^{\mc{LT}_i-\mathrm{an}}(\mc O, K)$; as every embedding $\sigma: \mc O \rightarrow \mc O_{\mb C_p}$ is continuous, then the polynomials in $\sigma(z)$ are dense in   $C^{\mc{LT}_i-\mathrm{an}}(\mc O, K)$ too. Moreover, if $f \in \Lambda_{\mc{LT}_i}^\circ$ corresponds to a distribution $\mu$ such that $\mu(z^n)=0 $ for all $n \gg 0$, then by the estimate of the valuation of the leading coefficient of the $n$-th Mahler polynomial \cite[(1.5)]{UdiBordeaux} we get that $\mu$ must be $0$ (otherwise it would be unbounded). 
\end{enumerate}
\end{proof}
Everything we said above for characters of $\mc O^\ast$ generalises immediately to characters of $(\mc O^\ast)^n$, where $n \geq 0$ is any integer. We obtain in this way the spaces $\mc W_{\mc{LT}_i^n}^\circ$ and $\mc W_{\mc{LT}_i^n}$ and the rings $\Lambda_{\mc{LT}_i^n}^\circ$ and $\Lambda_{\mc{LT}_i^n}$. Similarly for $\mc D(\mc{LT}_i^n)$. We are now ready to define the weight space we are interested in.
\begin{defi}\label{def:weights}
In case (L) then we define the weight space $\mc W=\mc W_G$ to be the rigid space given by
\begin{align*}
\prod_{i=0}^{f-1} \mc W_{\mc{LT}_i^{a_{i+1}-a_i}} \times \prod_{i=1}^{f}\mc W_{\mc D(\mc{LT}_i^{b_i-b_{i+1}})},
\end{align*}
under the convention that $b_{f+1}=0$.

In case (U) then we define the weight space $\mc W=\mc W_G$ to be the rigid space given by
\begin{align*}
\prod_{i=0}^{f-1} \mc W_{\mc{LT}_i^{a_{i+1}-a_i}}.
\end{align*}
\end{defi}
\begin{rmk} \label{rmq: no char et}
Note that $J(\mc{LT}_f)=\emptyset$ so $\mr{Hom}_{\mr{cont},\mc{LT}_f}(\mc O^\ast,\m C_p^\ast)=1$, the constant character. This gives another reason to discard the étale part in our Igusa tower: not only it has no $\mr{Fil}^1$, {\it i.e.} no non-zero differentials and in particular no non-zero modular forms, but also the corresponding weight space is trivial.
\end{rmk}
The following lemma tells us the our weight space has the expected dimension.
\begin{lemma} \label{lemma: dim weight space}
We have that $\mc W$ is equidimensional of dimension $nd$.
\end{lemma}
\begin{proof}
We can base change $\mc W$ to $\m C_p$ and use Proposition~\ref{prop: descr W}. In case (L) we get
\begin{gather*}
\dim(\mc W) = \sum_{i=0}^{f-1} (f-i)(a_{i+1}-a_{i}) + \sum_{i=1}^f i(b_{i}-b_{i+1}) =  \\
\sum_{i=0}^{f-1} (f-i)(a_{i+1}-a_{i}) + \sum_{i=1}^{f} i(a_{i+1}-a_{i}) = f \sum_{i=0}^f (a_{i+1}-a_{i}) = f(a_{f+1} -a_0)= fn =nd. 
\end{gather*}
In case (U) we get 
\begin{gather*}
\dim(\mc W) = \sum_{i=0}^{f-1} (f-i)(a_{i+1}-a_{i})= f a_1 + (f-1)a_2 -(f-1)a_1 +\cdots + a_f -a_{f-1} \\ =
a_1 +a_2 + \ldots + a_{f-1} +a_{f}= n(f/2)=nd. 
\end{gather*}
\end{proof}
In particular, these weights are the character of the intersection of $\mr{Gal}(\mr{Ig}/S^{\mr{tor}})$ with the $\mc O$-points of the parabolic $P$ of \cite[\S 1.1.2]{Moonen} (which is also the parabolic $P$ of Section \ref{sec:classicalforms}). Explicitly, it consists of  $\mc O$-linear maps $M_G \rightarrow M_G$ which preserve both the decomposition of $\mr{BT}_G$ into product of generalised Lubin--Tate groups and the $\mc O$-module $\mr{Fil}^{1}(M_G)$.
\begin{rmk}
For the reader who are unhappy about our weight space, we offer an alternative definition of the Igusa tower, more in the spirit of \cite{EischenMantovan,Valentinthese} but that can still be extended to the boundary. Consider the filtration of the $p$-torsion of the universal $\mu$-ordinary abelian variety over $S_m$ given by  
\begin{align*}
0 \subset \mc A[p^{\infty}]^{\mr{mult}} \subset \mc A[p^{\infty}]^{\circ} \subset \mc A[p^{\infty}],
\end{align*}
where $\mc A^{\mr{mult}}$ resp. $\mc A^{\circ} $ denotes the multiplicative resp. connected part. The $p$-divisible group $\mc A[p^{\infty}]^{\circ}/ \mc A[p^{\infty}]^{\mr{mult}}$ is naturally polarised. One could then define $\mr{Ig}_{m,l}$ as
\begin{align*}
 \mr{Isom}_{\mc O-\mr{BT}}(\mc A[p^{l}]^{\mr{mult}}, \mc O \otimes \mu_{p^{l}}) \times \mr{Isom}_{\mc O-\mr{pol}-\mr{BT}}\left(\frac{\mc A[p^{l}]^{\circ}}{ \mc A[p^{l}]^{\mr{mult}}}, \mr{BT}^{*}[p^l]\right),
\end{align*} 
where $\mr{BT}^{*}$ is the Barsotti--Tate with $\mc O$-structure isomorphic to $\mc A[p^{\infty}]^{\circ}/ \mc A[p^{\infty}]^{\mr{mult}}$. These $\mr{Ig}_{m,l}$ extend to $S_m^{\mr{tor}}$.
\end{rmk}
\begin{defi}\label{def:Lambda}
In case (L) we define our Iwasawa algebra as
\begin{align*}
\Lambda \colonequals & \mathop{\widehat{\bigotimes}}\limits_{i=0}^{f-1} \Lambda_{\mc{LT}_i^{a_{i+1}-a_i}} \ctens_{\mc O_K} \mathop{\widehat{\bigotimes}}\limits_{i=1}^{f}\Lambda_{\mc D(\mc{LT}_i^{b_i-b_{i+1}})},\\
\Lambda^\circ \colonequals & \mathop{\widehat{\bigotimes}}\limits_{i=0}^{f-1} \Lambda^\circ_{\mc{LT}_i^{a_{i+1}-a_i}} \ctens_{\mc O_K} \mathop{\widehat{\bigotimes}}\limits_{i=1}^{f}\Lambda^\circ_{\mc D(\mc{LT}_i^{b_i-b_{i+1}})},
\end{align*}
under the convention that $b_{f+1}=0$.

In case (U) then we define  our Iwasawa algebra as
\begin{align*}
\Lambda \colonequals \mathop{\widehat{\bigotimes}}\limits_{i=0}^{f-1} \Lambda_{\mc{LT}_i^{a_{i+1}-a_i}}, & \mbox{ and } \Lambda^\circ \colonequals \mathop{\widehat{\bigotimes}}\limits_{i=0}^{f-1} \Lambda^\circ_{\mc{LT}_i^{a_{i+1}-a_i}}.
\end{align*}
\end{defi}

\subsection{The Hodge--Tate map}\label{sec:HodgeTate}
We define ($p$-divisible) $p$-adic modular forms as elements of $\tilde V$, where 
\begin{align*}
\tilde{V} & \colonequals (\varinjlim_m\varinjlim_l V_{m,l})^{N_{\mr{Ig}}}, \\
V_{m,l} & \colonequals  \mr H^0(S^{\mr{tor}}_m,\mc O_{\mr{Ig}_{m,l}}) = \mr H^0(S^{*}_m,\mc O_{\mr{Ig^*}_{m,l}}).
\end{align*}
Here $N_{\mr{Ig}}$ is the unipotent part of $\mr{Gal}(\mr{Ig},S^{\mr{tor}})$. We have that $T_{\mr{Ig}}$ acts naturally on $\tilde V$. We want to interpret classical modular forms as $p$-adic modular forms.

Let $U$ be an open subset of the $\mu$-ordinary locus. Over $U$, given a trivialisation $\gamma\colon \mc {A}^{\mu}[p^{\infty}]^{\circ} \cong_U BT_G^{\circ}$ we can define a trivialisation of $\omega$. The filtration of Proposition \ref{Prop:MoonenFil} gives a sequence of quotients $\left\{ \omega_{\mc A_i} \twoheadrightarrow \omega_{\mc A_{i-1}} \right\}$. Over $S$ these sheaves are locally free, hence the quotients all split and the splitting can be moreover chosen to be $\mc O$-equivariant. It can be extended to the toroidal compactification as $\mc A_{W,i} / \mc A_{W,0}$ are independent of $W$. Hence over the $\mu$-ordinary locus we have an isomorphism
\[
\omega = \bigoplus_{i=1}^f \omega_{\mc A_{i}/\mc A_{i-1}}.
\]
Recall that $\mc A_{i+1}/\mc A_{i} \cong \mc{LT}_{i-1}^{a_i-a_{i-1}}$. As explained above in Subsection~\ref{subsec: dieud}, we have fixed a basis $(e_{\sigma,i})$ of the Dieudonné module of $\mc{LT}_{i-1}^{a_i-a_{i-1}}$. This gives a basis $(e_{\sigma,j})$ of $\omega_{ \mc{LT}_i^{a_i-a_{i-1}}}$. In particular $\left\{\gamma^{*}(e_{\sigma,j})\right\}$ is a basis for $\omega_{S^{\mr{tor}_{\infty}}}$. Recall the $P=\prod_{i=1}^d \mr{GL}_{a_{i}}\times \mr{GL}_{b_i}$-torsor $\mc E$. Summing up, we have a map 
\begin{align*}  
\mr{HT}:  \mr{Ig} \rightarrow \mc E \times S^{\tor}_{\infty}
\end{align*}
sending $(\mc A^{\mu}_x,\gamma)$ to $(\left\{\gamma^{*}(e_{\sigma,j})\right\},\mc A^{\mu}_x)$. 
\begin{defi} \label{defi: HT}
Given a classical modular form $f$, viewed as a function on $\mc E$, we obtain, by restriction and composition with $\HT$, the function $\HT^\ast(f)$ on $\mr{Ig}$. Since by definition $f$ is invariant under the action of $\prod_i N_{a_i} \times N_{b_i}$, we have that $\HT^\ast(f)$ is invariant under the action of $N_{\mr{Ig}}$, \emph{i.e.} we have
\[
\HT^\ast(f) \in \tilde V.
\]
\end{defi}
\begin{rmk}
The name Hodge--Tate could seem a bit out of turn: the map in  \cite{PilHida,hida_control} which compares classical and $p$-adic forms is induced by the Hodge--Tate map  $\mc A^{\vee} \rightarrow \omega_{\mc A}$. In particular, the $\mr{HT}$ maps in {\it loc. cit.} would be ours composed with an involution.
\end{rmk}
We can describe the HT map more precisely: locally it sends $\mr{Gal}(\mr{Ig}/S^{\mr{tor}})$ into $\prod_{i=1}^d \mr{GL}_{a_{i}}\times \mr{GL}_{b_i}(\mc O)$ and we can write down explicitly this map. In case (U), remember that $f=2d$ and that we have written $a_{j+d}$ for  $b_{d-j+1}$. Let $(\gamma_i)_{i=0,\ldots, f-1}$ be an element of $\prod_{i=0}^{f-1} \mr{GL}_{a_{i}-a_{i-1}}(\mc O)$. We have explicitly 
\begin{align*}
\mr{HT}({(\gamma_i)}_{i=0,\ldots, f-1})= \prod_{j=1}^d \left([\sigma_{j}(\gamma_i)]_{i=0}^{j-1},[\sigma_{f+1-j}(\gamma_i)]_{i=0}^{f-j}\right), 
\end{align*} 
where $[\sigma_{j}(\gamma_i)]_{i=0}^{j-1}=[\sigma_{j}(\gamma_0),\dots,\sigma_{j}(\gamma_{j-1})]$ is an element of $\prod_{i=0}^{j-1} \mr{GL}_{a_{i+1}-a_{i}}(\mc O)$ embedded diagonally into $\mr{GL}_{a_{j}}(\mc O)$. 

In case (L) the description is similar; let $ \gamma^+=(\gamma^+_i)_{i=0,\ldots, f-1}$ and $\gamma^-=(\gamma^-_i)_{i=1,\ldots, f}$ be an element of $\prod_{i=0}^{f-1} \mr{GL}_{a_{i}-a_{i-1}}(\mc O) \times \prod_{i=1}^{f} \mr{GL}_{b_{i}-b_{i+1}}(\mc O)$. We have explicitly 
\begin{align*}
\mr{HT}(\gamma^+,\gamma^-)= \prod_{j=1}^d \left([\sigma_{j}(\gamma^+_i)]_{i=0}^{j-1},[\bar{\sigma}_{j}(\gamma^-_i)]_{i=j}^{f}\right).
\end{align*} 

\begin{esempio}
When $G$ is as in Example \ref{esempio2} with $f=4$ we have 
\begin{align*}
& \mr{HT}(\gamma_0,\gamma_1,\gamma_2,\gamma_3)=  \\
& = \left(
 \left(\sigma_1(\gamma_0), \left(\begin{array}{cccc}
\sigma_4(\gamma_0)& & & \\
 & \sigma_4(\gamma_1) & & \\
& & \sigma_4(\gamma_2) & \\
& & & \sigma_4(\gamma_3)
\end{array}\right) \right) \right.,\\
& 
\left. \left( \left(\begin{array}{cc}
\sigma_2(\gamma_0)&  \\
 & \sigma_2(\gamma_1) 
\end{array}\right), 
\left(\begin{array}{ccc}
\sigma_3(\gamma_0)& &  \\
 & \sigma_3(\gamma_1) &  \\
& & \sigma_3(\gamma_2) 
\end{array}\right) \right) \right) 
\end{align*}

\begin{rmk}
This embedding defines a map from set the algebraic weights of the parabolic $P$ of Section \ref{sec:classicalforms} to our $p$-adic weight space.
\end{rmk}
Note that, without loss of generality, we can assume $\sigma_4=\mr{Id}$. (This goes back to the choice of two fixed embeddings of $\overline{\m Q}$ in $\m C$ and $\m C_p$). We hope that this also clarifies our choice of the weights in $\mc W$.
\end{esempio}
We define by $\mr{Iw}$ the Iwahori subgroup of $G$ defined as the counter image of the Borel of $G({\mc O/p})$ (which is quasi-split as every algebraic group over a finite field) and by $\G_1(p)$ the counter image of the corresponding unipotent. More generally, we denote by $\G_0(p^l)$ (resp. $\G_1(p^l)$) the subgroup of $G(\m Z_p)$ whose elements, seen inside $G(\mc O_F)$ (which is $\cong \prod_{i=1}^{f/2} \mr{GL}_n(\mc O_F) \times  \mr{GL}_n (\mc O_F)$ in case (U) and $\cong \prod_{i=1}^{f} \mr{GL}_n(\mc O_F) \times  \mr{GL}_n(\mc O_F)$ in case (L))  reduce modulo $p^l$ to a matrix in the Borel (resp. unipotent) subgroup of $P(\mc O_F/p^l)$. 

Denote by $\mr M_{\kappa}\left(\mc H \cap\Gamma_1(p^l),\varepsilon,\mc O_K\right)$ the space of classical modular forms, defined over $\mc O_K$ (where $K$ is our coefficient field), of weight $\kappa$, level $\mc H \times \Gamma_1(p^l)$ and transforming under the finite order character $\varepsilon$ of $\G_0(p^l)/\G_1(p^l)$.
We conclude with the following proposition:
\begin{prop}\label{prop:classicaltopadic}
We have a map 
\begin{align}\label{eqn:classicaltopadic}
\mr M_{\kappa}\left(\mc H \cap\Gamma_1(p^l),\varepsilon,\mc O_K\right)\otimes \mc O_K/p^m \rightarrow V_{m,l}^{N_{\mr{Ig}}}.
\end{align} 
\end{prop}
\begin{proof}
The Hodge--Tate map gives a natural map 
\begin{align*}
\mr M_{\kappa}\left(\mc H \cap\Gamma_1(p^l),\varepsilon,\mc O_K \right)\otimes \mc O_K/p^m \rightarrow \mr H^0(\mr{Ig}_{m,l},\omega^{\kappa})
\end{align*} 
given by restriction to the $\mc O/ p^m$-points to the subgroup described above. Note that we have a linear map $l_{\mr{can}}:R[
\kappa] \rightarrow \m A^1$ which consists to evaluation at $1$. This gives 
\begin{align*}
l_{\mr{can}}: \mr H^0(\mr{Ig}_{m,l},\omega^{\kappa}) \rightarrow \mr H^0(\mr{Ig}_{m,l},\mc O_{\mr{Ig}_{m,l}})= V_{m,l}.
\end{align*} 
By the definition of $\G_1(p^l)$ the image is invariant under the $N_{\mr{Ig}}$.
\end{proof}

\section{Hecke operators}\label{sec:Heckeoperators}
As in \cite{hida_control}, we introduce two types of $\mathbb{U}_p$'s operators. Some representation-theoretic operators on the Igusa tower that shall be used to define an ordinary projector $e_{\mr{GL}}$, and some geometric operators coming from the Shimura variety of Iwahori level for the projector $e_G$. We shall compare these two operators at the end of the section. 

\subsection{Hecke operators on the Igusa tower}
We define $\mr{GL}$-type Hecke operators following \cite[\S 2.6]{hida_control}. Let $\mr{GL}_m$ be the usual algebraic group, $B_m$ the Borel of upper-triangular matrices, $T_m$ its maximal torus and $U_m$ the unipotent part. For $j=1,\dots, m$ we let $\alpha_j$ be the diagonal matrix $[1,\dots,1,p,\ldots p]$ with $j$ $p$'s and we let $t_j$ be the double coset $U_m(\m Z_p) \alpha_j U_m(\m Z_p)$. We can write as usual $t_j = \bigsqcup U_m(\m Z_p) u $, where $u$ varies in a set of representatives of the quotient $t_j/U_m(\Z_p)$. If $f $ is a function on $\mr{GL}_m(\mc O/ p^l \mc O)$ invariant by the action of $B_m(\mc O/ p^l \mc O)$, we let $t_j$ acts on $f$ via the formula 
\begin{align*}
f \vert t_j (\gamma) = \sum_u f(u.\gamma). 
\end{align*}
where $u.\gamma$ is defined as $\gamma^{-1} u^{-1} \gamma$ and the sum is independent of the choice of the set of representatives. If $f$ is invariant only by $U_m(\mc O/ p^l \mc O)$ but is in the $\kappa$-eigenspace, for $\kappa$ a character of $T_m(\mc O/ p^l \mc O)$, then $f \vert t_j $ is also well-defined and belongs to the same eigenspace for $T_m(\mc O/ p^l \mc O)$. We define $\mb t_m \colonequals \prod_{j=1}^m t_j$. Note that $\prod_{j=1}^m \alpha_i$ contracts $U_m(\mc O)$ into $\mr{Id}_m$. Let $A$ be a $p$-adically complete flat $\Z_p$-algebra, and let $K \colonequals A[1/p]$. We define $R_A[\kappa]$ (resp. $F_K[\kappa]$) as the algebraic (resp. topological, in the sense of \cite[\S 3.1.2]{PilHida}) representation of $\mr{Res}^{\mc O}_{\m Z_p}\mr{GL}_m$ (resp. of $\mr{GL}_m(\mc O)$) of highest weight $\kappa$ for $\mr{Res}^{\mc O}_{\m Z_p} T$ (resp. $T_m(\mc O)$). If $\kappa$ is regular (resp. very regular) then $\mb t^{n!}_m$ converges to an operator $e_m$. The image $e_m(R_A[\kappa])$ (resp. $e_m(F_K[\kappa])$) is isomorphic, via the projection $l_{\mr{can}}$, to the line generated by the highest weight vector (see \cite[Proposition 3.2, Proposition 3.3]{PilHida}).

We are ready to define the ordinary projector $e_{\mr{GL}}$.
\begin{defi} \label{defi: ord proj}
In case (U) we define
\begin{align*}
\mb t_{\mr{GL}} \colonequals  \prod_{i=0}^{f-1} \mb t_{a_{i+1}-a_i}.
\end{align*}  
and in case (L) we define
\begin{align*}
\mb t_{\mr{GL}} \colonequals  \prod_{i=0}^{f-1} \mb t_{a_{i+1}-a_i} \times \prod_{i=1}^f  \mb t_{b_{i}-b_{i+1}}.
\end{align*}
Since the Igusa tower is a cover of $S_\infty^{\tor}$ with group $\prod_{i=0}^{f-1} \mr{GL}_{a_{i+1}-a_{i}}(\mc O)$ in case (U) and $\prod_{i=0}^{f-1} \mr{GL}_{a_{i+1}-a_{i}}(\mc O) \times \prod_{i=1}^{f}\mr{GL}_{b_{i}-b_{i+1}}(\mc O)$ in case (L), we obtain an operator
\begin{gather*}
\mb t_{\mr{GL}} \colon \tilde{V} \to \tilde{V} \\
f(x,\gamma) \mapsto f(x, \mb t_{\mr{GL}}(\gamma)).
\end{gather*}
Here $x$ corresponds to a point of $S_\infty^{\tor}$ and $\gamma$ is a trivialisation as in the definition of the Igusa tower.

We have that $\mb t_{\mr{GL}}$ preserves $\tilde{V}[\kappa]$, for $\kappa$ a character of $T_{\mr{Ig}}$. Moreover, with the same reasoning as  \cite[page~20]{hida_control}, we see that it sends $V_{m,l+1}[\kappa]$ into $V_{m,l}[\kappa]$. In case (U), whenever it is well-defined (it is necessary to know that the module spanned by $\mb t^{n}_{\mr{GL}}(f)$ is of finite rank for all $f \in V_{m,l}$), we define
\[
e_{\mr{GL}}\colonequals  \lim_{n_0,\ldots,n_{f-1}} \mb \prod_{i=0}^{f-1} \mb t^{n_i!}_{a_{i+1}-a_i}.
\]
Similarly for case (L).
\end{defi}
\subsection{Hecke operators on the Shimura variety}
In this section only we will work over the rigid fiber of our varieties, without changing the notation. We recall the definition of the correspondences $C_i$ used to define the $\mr U_{p,i}$ operators from \cite[\S 2.3]{Bijamu}. He defines correspondences over (the rigid fiber of) $X_{\mr{Iw}}$.

In case (L), let $ 1 \leq  i \leq n-1$. The correspondence $C_i$ parametrises points $(A, \lambda, \iota, \alpha, (H_{i,\bullet})_i)$ of $X_{\mr{Iw}}$ with in addition an $\mc O$-stable subgroup $L = L^+ \oplus L^- $ for which $H_i \oplus L^+ = A[\pi^+]$. 
By Remark \ref{rmk: orth} it is enough to specify $L^+$. 
We have two projections $p_1,p_2:C_i \rightarrow X_{\mr{Iw}}$ corresponding to forgetting $L$ and taking the quotient by $L$. 
The universal isogeny $\mc A \rightarrow \mc A/L$ induces an isomorphism $p^{*}(\kappa):p_2^* \omega^{\kappa} \rightarrow p_2^*\omega^{\kappa}$.
One defines $\tilde{\mr{U}}_{p,i}$ as the composition 
\begin{align*}
\mr H^0(X_{\mr{Iw}},\omega^{\kappa}) \rightarrow \mr H^0(C_i,p_2^*\omega^{\kappa}) \stackrel{p^*(\kappa)}{\longrightarrow} \mr H^0(C_i,p_1^*\omega^{\kappa})\stackrel{\mr{Tr}(p_1)}{\longrightarrow} \mr H^0(X_{\mr{Iw}},\omega^{\kappa}).
\end{align*}
Define moreover: 
\begin{align*}
N_i = & \sum_{j}  \mr{max}(a_{j}-i,0) \kappa_{\sigma_j,a_j} +  \mr{max}(j-a_{j},0)\kappa_{\bar \sigma_j,b_j},\\
n_i = & \sum_{j}  \mr{min}(a_{j},j) \mr{min}(n-i,b_j).
\end{align*}
The number $N_i$ comes from $p^*(\kappa)$ while $n_i$ is an inseparability degree. 
\begin{defi}\label{def:U_p}
 We define $\mr U_{p,i}\colonequals p^{-N_i-n_i}\tilde{\mr{U}}_{p,i}$;  it is well defined over $\mc O_K$ (see Remark \ref{rmk:welldefined}).
\end{defi}
In case (U) let $ 1 \leq i \leq n/2 $. If $i < n/2$ the correspondence $C_i$ parametrises points $(A, \lambda, \iota, \alpha, (H_{i,\bullet})_i)$ of $X_{\mr{Iw}}$ with in addition an $\mc O$-stable subgroup $L$ of $\mc A[p^2]$ such that $H_i \oplus L[p] = A[p]= {H_i}^{ \perp} \oplus pL $.
If $i=n/2$ one considers $L$ such that  $H_i \oplus L = A[p]$.

Define:
\begin{align*}
N_i = &  \sum_{j}  \mr{max}(a_{j}-i,0) \kappa_{\sigma_j,a_j} +  \mr{max}(b_j-i,b_j-a_j)\kappa_{\bar \sigma_j,b_j}  {\mbox{ if }i < n/2 },\\
n_i = & n \sum_{j}  \mr{min}(a_{j},i)  {\mbox{ if }i < n/2 },\\
N_i = & \sum_{j}  \frac{b_j-a_j}{2}\kappa_{\bar \sigma_j,b_j}  {\mbox{ if }i = n/2 },\\
n_i = & \frac{n}{2}\sum_{j}a_j {\mbox{ if }i = n/2 }.
\end{align*}

The definitions of $\tilde{\mr U}_{p,i}$ and ${\mr U}_{p,i}$ are as before.

\subsection{A comparison}\label{sec:comparison}
We now want to extend the action of the operators $\mr U_{p,i}$ to the Igusa tower. 
The idea is that the filtration defined by Proposition \ref{Prop:MoonenFil} will give a section to $X_{\mr{Iw}}$.

The first thing to do is to interpret these operators from a linear algebra perspective. Let $k = \overline{\m F}_p$. We can identify the Dieudonn\'e module of $\mc A^{\mu}_{/k}$ with the Tate module of  $\mc A^{\mu}_{/\mr{Frac}(W(k))}$. We recall that the Tate module has rank $2dn$ over $\m Z_p$, and that $f=d$ in case (L) and $f=2d$ in case (U).

The Hecke operators in case (L) are the same as the ones of \cite[\S 6.5]{hida_control}. We let  $\langle\overline{x}_1,\ldots,\overline{x}_{n} \rangle$ be an $\mc O$-basis for $\mc A^{\mu}_{/\mr{Frac}(W(k))}[\pi^+]$ and $\langle \overline{x}_{n+1},\ldots,\overline{x}_{2n} \rangle$ be an $\mc O$-basis for $\mc A^{\mu}_{/\mr{Frac}(W(k))}[\pi^-]$. We say that the basis is adapted to $L_i$ if $L_i$ can be written as
\[
L_i = \langle \overline{x}_{i+1},\ldots,\overline{x}_{n}, \overline{x}_{2n-i},\ldots,\overline{x}_{2n}\rangle.
\]
The $\mr U_{p,i}$ operator can be interpreted as the double quotient operator $\mr{Iw}\beta_i \mr{Iw}$. Let us describe the corresponding matrix $\beta_i$: if $i>a_j$ the $j$-th component of $\beta_i$ is $\left( [p,\ldots,p],[1,\ldots,1,p,\dots, p]\right)$ where $p$ appears $a_j$ times in the first matrix and $i-a_j$ times in the second; if $a_j \geq i$ then it is $\left( [1,\ldots, 1,p,\ldots,p],[1,\ldots,1]\right)$ where $p$ appears $i$ times. The number $n_i$ defined before coincides with the index defined by Hida, that he calculates explicitly in \cite[(6.12)]{hida_control}.

In case (U), fix a basis $\langle x_1,\ldots,x_{n} \rangle$ of the Tate module as $\mc O$-module adapted to  $L_i$, and let $\langle \overline{x}_1,\ldots,\overline{x}_{n} \rangle$ be the corresponding basis for $\mc A^{\mu}_{/\mr{Frac}(W(k))}[p^2]$.  This means that $L_i \subset \mc A^{\mu}_{/\mr{Frac}(W(k))}[p^2]$ is given by 
\begin{align*}
\langle p\overline{x}_{i+1},\dots, p\overline{x}_{n-i}, \overline{x}_{n-i+1},\ldots, \overline{x}_{n} \rangle
\end{align*}
if $i < n/2$. If $i=n/2$ then $L = \langle p\overline{x}_{n/2},\dots, p\overline{x}_{n} \rangle \subset \mc A^{\mu}_{/\mr{Frac}(W(k))}[p]$. We can define for each $L_i$ an adelic matrix $\beta_i $. If $i < n/2$  and $a_j \geq i$ then
\[
\beta_i=[1,\ldots,1,p,\dots, p, p^2,\ldots,p^2]
\]
with $i$ $1$'s and $p^2$'s, and $n-2i$ $p$'s. At places where $ i >a_j$ we have
\[
\beta_i = [1,\ldots,1,p,\dots, p, p^2,\ldots,p^2]
\]
where $1$ appears $a_j$ times and $p$ appears $n-2a_j = b_j -a_j$ times. If $i=n/2$ then $\beta_i = [1,\ldots,1,p,\ldots,p]$ with $n/2$ $1$'s and $p$'s.
\begin{rmk}
Note that the number $n_i$ is the sum over all places of $F_0$ of the numbers defined in \cite[(6.10)]{hida_control}. This is not really clear at first sight so let's calculate them. Suppose $i \neq n/2$. Build a $a_j \times b_j$ table and put a $+1$ in the rows from $i+1$ to $a_j$ (accounts for the multiplication by $\beta_i$), then add a $-2$ in the columns from $b_j-i+1$ till $b_j$ and a $-1$ in the column from $b_j-(a_j+b_j -2i+i)=i-a_j$ till $b_j-i-1$ (these account for the multiplication by $\beta_i^{-1}$). 

If $a_j \leq i$ then there are no $+1$, $-1$ appears $a_j(b_j -i-1-(i-a_j-1)))=a_j(n-2i)$ times and $-2$ appears $a_j(b_j-(b_j-i+1-1))= a_j i$  times. Sum up the opposite to get $a_j(n-2i) + 2a_j i= na_j$. Similar if $a_j >i$ or $i=n/2$.
\end{rmk}
Concerning the explanation for $N_i$, note that the matrix $\beta_i$ is sent diagonally into \begin{align}\label{eqn:GoverO}
 G(\mc O \otimes_{\m Z_p} \mc O ) \cong \prod_{j=1}^d ( \mr{GL}_{a_j}\times \mr{GL}_{b_j}) \times (\mr{GL}_{b_j} \times \mr{GL}_{a_j})  (\mc O)
\end{align} and  only the first factors of the two couples contribute (because they are the ones appearing in $P$). (So, for example in case (U), to justify the appearance of $\mr{max}(b_j-i,b_j -a_j)$, note that it is $b_j -i=(a_j-i)+(b_j-a_j)$ if $a_j \geq i$ and otherwise it is $b_j -a_j$, which is exactly  the number of times that $p$ appears in the two upper blocks  $( \mr{GL}_{a_j}\times 1_{b_j}) \times (\mr{GL}_{b_j} \times 1_{a_j}) $.) 

In both case, one can define an operator $\mb T_{i}$  on $V_{m,l}[\kappa]$, similarly to \cite[page~36]{hida_control}.  For each point $(\mc A_x,\gamma)$ where $\mc A_x$ is the $\mu$-ordinary abelian variety corresponding to $x$ in $S_m$ and $\gamma$ a class in $\mr{Gal}(\mr{Ig}_{m,l}/S^{\mr{tor}}_m)/N_{\mr{Ig}}(\mc O/p^l \mc)$ and for each isogeny $\xi: \mc A_x \rightarrow \mc A_x/L_x$ one defines a point $(\mc A_x/L_x,\gamma_{L_x})$ in  $\mr{Ig}_{m,l}$. 
\begin{defi}\label{def:T_i}
We define the geometric operator $\mb T_i$ as follows:
\begin{align*}
\mb T_{i} f (x,\gamma)= p^{-n_i}\sum_{ L_x \in p_2p_1^{-1}(x) } f(\mc A_x/L_x,\gamma_{L_x}).
\end{align*}
\end{defi}
\begin{rmk}\label{rmk:welldefined}
Note that this expression is well defined as the na\"ive operators are divisible by $n_i$ by \cite[Lemma 6.6]{hida_control}. 
\end{rmk}
We define $\mb T \colonequals  \prod_{i=1}^{n/2} \mb T_{i}$.
Assuming it converges, we can hence define $e_G=\lim_n \mb T^{n!}$.
From the matrix description and the decomposition in \eqref{eqn:GoverO} it is clear that the projector $e_{\mr{GL}}$ is contained in $e_G$. It is clear that we have the relation $ \mb T_i = \mr U_{p,i}$ (after composition with \eqref{eqn:classicaltopadic}) hence the operators of the previous section are well defined.

To conclude, we make the following observation. 
\begin{lemma}\label{lemma:commute}
Let $\mc A^{\mu}$ be the $\mu$-ordinary universal abelian variety over $\overline{\m F}_p$. For every $L_i$ we have $\mr{Ha}^{\mu}(\mc A^{\mu}) = \mr{Ha}^{\mu}(\mc A^{\mu}/L_i) $.  In particular $$\mr U_{p,i}(\mr{Ha}^{\mu}f)=\mr{Ha}^{\mu} \mr U_{p,i}(f)$$
and moreover  we can normalise $\mr{Ha}^{\mu}$ so that $\mr{Ha}^{\mu} \equiv 1$ on $S^{\mr{tor}}$.
\end{lemma}
\begin{proof}
Using the Dieudonn\'e functor, we see that the quotient $\mc A^{\mu} \rightarrow \mc A^{\mu}/L_i $ corresponds in case (U) to the sub-lattice
\[
M'\colonequals \langle x_1,\ldots,  x_i, px_{i+1},\ldots, px_{n-i}, p^2 x_{n-i+1},\dots, p^2 x_{n} \rangle
\]
of $M\colonequals \langle x_1,\ldots,x_{n} \rangle$. By the definition of the Hasse invariant, we are calculating the determinant of ${\bigwedge^{a_i} \mr{Fr}^{f}}$, suitably normalised. Moreover the filtration on $M'$ is induced by the one of $M$, hence the determinant is the same. This tells us also that $\mr{Ha}^{\mu} $ is constant on $S^{\mr{tor}}$. The same holds for case (L). 
\end{proof}

\subsection{\texorpdfstring{Hecke operators at $p$ without level at $p$}{Hecke operators at p without level at p}}

Let $X$ be the compact Shimura variety of level $\mc H$. Let $\beta_i$ the matrix associated with $\mr U_{p,i}$ in the previous section. 
\begin{defi} 
We define the Hecke operator $\mr T_{p,i}$ on $\mr M_{\kappa}(\mc H, R)$ using the double quotient $[G(\m Z_p) \beta_i G(\m Z_p) ]$.
\end{defi} 
We are interested in this opersator as if  $\kappa \gg 0$ then we have, by the same argument of \cite[pp.~467-468]{hida_smf}, that 
\begin{align}\label{eq:TpUp}
\mr T_{p,i} \equiv \mr{U}_{p,i} \bmod p .
\end{align}
\subsection{\texorpdfstring{Prime-to-$p$ Hecke operators}{Prime-to-p Hecke operators}}
We define now the Hecke algebra. Let $N$ be a prime-to-$p$ integer containing all prime numbers which are norms of prime ramified in $F$ or for which $G$ is not quasi-split and let $\mathfrak{l}$ be a prime ideal of $F_0$ above $l \nmid pN$, we define
\begin{align*}
\m T_{G,\mathfrak{l}} = \Z_p[ G(\mc O_{F_{0,\mathfrak{l}}}) \setminus G(F_{0,\mathfrak{l}})/G(\mc O_{F_{0,\mathfrak{l}}})  ].
\end{align*}
If $\mathfrak{l}$ is split in $F$, then $G(F_{0,\mathfrak{l}})\cong \GL_{a+b}(F_{0,\mathfrak{l}})\times \m G_m(\Q_l) $ and $\m T_{G,\mathfrak{l}}$ is generated by the same matrices. If $\mathfrak{l}$ is inert, $G(F_{0,\mathfrak{l}})$ is contained in $\GL_{a+b}(F_{\mathfrak{l}})$ and generated by the same diagonal matrices.

The abstract Hecke algebra of prime-to-$Np$ level is then
\begin{align*}
 \m T_G^{(Np)} = {\bigotimes_{\mathfrak{l} \nmid Np}}\!\!' \, \m T_{G,\mathfrak{l}}.
\end{align*}
It acts naturally on the space of $p$-adic modular forms. 

\section{Hida theory}\label{sec:Hidatheory}
In this section we do not suppose that $p$ is inert in $F_0$. The $p$-torsion of $\mc A^{\mu}$ decompose accordingly to the decomposition of $p$ in $F_0$, and for each  $\pi_i$ dividing $p$ we can define an Igusa tower $\mr{Ig}_{\pi_i}$, a weight space $\mc W_{\pi_i}$ etc. as in Section \ref{sec:Igusatower}, and we denote by $$\mr{Ig}= \mr{Ig}_{\pi_1} \times_{S^{\mr{tor}}} \mr{Ig}_{\pi_2} \times \cdots \times_{S^{\mr{tor}}}\mr{Ig}_{\pi_k},$$ and the same for  $\mc W^\circ$, $\mc W$, $\Lambda^{\circ}$, and $\Lambda$. We denote by $L$ the Galois group of $\mr{Ig} \rightarrow S_{\infty}^{\mr{tor}}$.

We also have the two projectors $e_{G,\pi_i}$ and $e_{\mr{GL},\pi_i}$. We set 
\begin{align*}
e_{G}=\prod_{\pi_i \mid p \text{ in }F_0}e_{G,\pi_i}, \;\;\;\; e_{\mr{GL}}=\prod_{\pi_i \mid p \text{ in }F_0} e_{\mr{GL},\pi_i}.
\end{align*}
We are almost ready to develop Hida theory but first we need some propositions. Recall that $D$ is the divisor of $X^{\mr{tor}} \setminus X$ and $\pi: X^{\mr{tor}} \rightarrow X^* $.

\begin{prop}\label{prop:sesmodp}
We have a short exact sequence: 
\begin{align*}
 0 \rightarrow \mr H^0(S^*, \pi_*(\omega^{\kappa}(-D))) \stackrel{\times p^m}{\rightarrow} \mr H^0(S^*, \pi_*(\omega^{\kappa}(-D)) \rightarrow \mr H^0 (S^*_m,\pi_*(\omega^{\kappa}(-D))) \rightarrow 0.
\end{align*}
If $t \in \mathbb{N}$ is big enough, we also have 
\begin{align*}
\mr H^0(X^*,\pi_*\omega^{\kappa+t(1,\ldots,1)}(-D)) \rightarrow \mr H^0(X_1^*,\pi_*\omega^{\kappa+t(1,\ldots,1)}(-D)).
\end{align*}
\end{prop}
\begin{proof}
We first check that we have a short exact sequence 
$$ 0 \rightarrow \pi_*(\omega^{\kappa}(-D)) \stackrel{\times p^m}{\rightarrow}  \pi_*(\omega^{\kappa}(-D)\otimes \mc O_K/ p^m  \rightarrow 0.$$
We use the $q$-expansion as described in \cite[Proposition 1.3]{BrasRos} at a point $x \in X^*$:  
\begin{align*}
\prod_{h} {\Homol^0(\mc Z,\mc L(h)\otimes \omega^k(-D))}^{\G(h)},
\end{align*} 
where $h$ runs over the element of maximal rank in the abelian part of the unipotent for $G$,  $\mc Z$ is isogenous to a power of a universal abelian variety whose dimension depends on $x$ and $\G(h)$ is the stabiliser of $h$ inside $\mc H$.   As we have chosen a neat level for $X$ and $h$ has maximal rank, then all stabilisers  $\G(h)$ are trivial. We are left to use the exact sequence of sheaves
\begin{align*}
0 \rightarrow \mc L(h)\otimes \omega^k(-D) \stackrel{\times p^m}{\rightarrow} \mr \mc L(h)\otimes \omega^k(-D) \rightarrow \mc L(h)\otimes \omega^k(-D)_m \rightarrow 0.
\end{align*}
Using Mumford vanishing theorem \cite[\S III.16]{Mumford} for the $\mr H^1$ of $\mc L(h)$ (remember that $h$ has maximal rank) we can conclude.

For the first part, we use the fact that $R^1\pi_* \mc O_{X^{\mr{tor}}(-D)}=0$ and that $S^*$ is affine.
For the second part, we need to check that $H^1(X_1^*,\pi_*\omega^{\kappa}(-D))$ vanishes. We can write 
$$ \pi_*\omega^{\kappa+t(1,\ldots,1)}(-D) \cong \pi_*(\omega^{\kappa}(-D))  \otimes {\omega^\star}^{\otimes t},$$ 
where $\omega^\star$ is the ample line bundle $\mc O(1)$ on the minimal compactification \cite[Theorem 7.2.4.1 (2)]{lan}. If $t$ is big enough by Serre's vanishing theorem we get 
\[
\mr H^1(X_1^*,\pi_*\omega^{\kappa+t(1,\ldots,1)}(-D))=0.
\]
\end{proof}
The proposition can fail whenever one consider non-cuspidal forms. This is because there are some restrictive conditions on the weight to have non-cuspidal forms \cite[Theorem 1.12]{BrasRos} in characteristic zero, while the same is not true modulo $p$ (imagine a weight which is not parallel but it becomes modulo $p$). Nevertheless, one can hope to develop Hida theory after restricting to weights for which there is no {\it a priori} obstruction to the existence of Eisenstein series. The interested reader is referred to \cite{LiuRos} where this strategy has been exploited to construct non-cuspidal Hida families of Siegel modular forms. 
\begin{prop}\label{prop:zerooutside}
If $\kappa$ is big enough (meaning $\kappa_{\sigma_j,b_j} >C$ and $\kappa_{\sigma_j,a_j} >C$ for all and $C$ depending only on the Shimura datum), then for all $f \in \mr H^0(X^{\mr{tor}},\omega^{\kappa})$ we have
\[
\mr \prod_i \mr U_{p,i} f (x) \equiv 0 \bmod p
\]
for all $x \in X^{\mr{tor}} \setminus S$.
\end{prop}
\begin{proof}
We mimic the proof for the Siegel case, as given in \cite[Proposition A.5]{PilHida}. Suppose not all the $a_j$'s are equal (otherwise we have a non-empty ordinary locus). Let $\mr{NP}$ be the $\mu$-ordinary Newton polygon and $\mr{NP}_j$ the Newton polygon which is above $N$ only at the point from  $a_j$ to $a_{j'}$ (where $a_j =\cdots = a_{j'}$) (like in \cite[Figure 7]{Valentinhasse}, but we admit that more consecutive slopes could be the same). 

Let $\mc N_j$ be the corresponding Newton stratum in the open Shimura variety and $x$ any point in it. The corresponding $p$-divisible group $\mc A_x[p^{\infty}]$ admits a filtration $\left\{\mc A_l\right\} $ for $l=0,\ldots,j-1,j'+1,f+1$. 

As the Newton polygon $\mr{NP}_j$ is above $\mr{NP}$, the slopes on $H\colonequals \mc A_{j'+1}/\mc A_{j-1}$ are bigger then the ones of the $j'+1$ filtered piece of $\mc A^{\mu}$. We need to understand (a bit) what is the matrix $M$ of first filtered piece of the Dieudonn\'e module of $\mc A_x / L$ w.r.t the first filtration of the Dieudonn\'e module of $\mc A_x$, for a subgroup that appear in the correspondence $C_{i}$ for a certain $i$ such that $L  \cap H \neq 0$.  In this case the matrix $M$ is gonna differ, at places $\sigma_k$ with $j \leq k \leq j' $ from the $k$-th component of $ \beta_i$, the matrix defined in Section 
\ref{sec:comparison} at least in one entry, where a $1$ is substituted by a $p^{\eps}$, with $0 <\eps <1$ (this number is related to the $\sigma_k$ degree of $H$, in the sense of Fargues \cite{fargues_can}). 

A similar calculation to the one that led to the choice of the number $N_i$ in \cite[Definition 2.4, 2.9]{Bijamu} tells us that the image of the map $p^*(\kappa)$ is gonna be contained inside $p^{M_i(x)} e^*\Omega^1{\mc A_x/ L/W(k)}$ where 
\begin{align*}
M_i(x) = N_i + n_i + \sum_k \eps_k \kappa_{\sigma_k, a_k},
\end{align*}
with at least one $\eps_k >0$. If $\kappa_{\sigma_k, a_k} > 1/\eps_k$ then $M_i -N_i -n_i \geq 1$, {\it i.e.} $p \mid \mr U_{p,i}f (x)$. As the reunion of $\mc N_j$ over the $j's$ is dense in $X \setminus S$ \cite[Theorem 1.1]{hamacher} we can conclude. 
\end{proof}

We conclude with the following very important theorem.
\begin{teo}\label{thm:boundedness}
We have that
\[
\dim_{\mc O_K/p}({\mr H^0(X^{\mr{tor}}_1,\omega^{\kappa}(-D))}^{\mr{ord}})  \mbox{ and } \rk_{\mc O_K}({\mr H^0(X^{\mr{tor}},\omega^{\kappa}(-D))}^{\mr{ord}})
\]
are bounded independently of $\kappa$.
\end{teo}
\begin{proof}
First note that if $f$ is ordinary then $\mr{Ha}^{\mu}f$ is ordinary too by Lemma \ref{lemma:commute}. If the weight is big enough then by Proposition \ref{prop:zerooutside} $\mr U_p f$ vanishes outside the ordinary locus, and so does $f$; thank to Theorem \ref{teo:appendix} we can choose $\mr{Ha}^{\mu}$ so that it has simple zeros and hence we can divide $f$ by $\mr{Ha}^{\mu}$.
Multiplication by powers of the Hasse invariant induces 
\begin{align*}
{\mr H^0(X^{\mr{tor}}_1,\omega_1^{\kappa})}^{\mr{ord}} \stackrel{\cong}{\rightarrow} {\mr H^0(X^{\mr{tor}}_1,\omega_1^{\kappa'})}^{\mr{ord}}.
\end{align*}
This implies that the space of ordinary modular forms modulo $p$ is bounded independently of $\kappa$.  Moreover, as $\kappa$ is such that the congruence \ref{eq:TpUp}  holds, then we can write (where the reunion ranges over a finite number of possible different reduced Hasse invariants):
\begin{align*}
{\mr H^0(S^{\mr{tor}}_1,\omega^{\kappa}(-D))}^{\mr{ord}} & \stackrel{\cong}{\rightarrow} {\left( \bigcup_{r} \frac{\mr H^0(X^{\mr{tor}}_1,\omega^{\kappa'+r w(\mr{Ha}^{\mu})}(-D))}{{\mr{Ha}^{\mu}}^r } \right)}^{\mr{ord}} \\
& \stackrel{\cong}{\rightarrow} {\mr H^0(X^{\mr{tor}}_1,\omega_1^{\kappa}(-D))}^{\mr{ord}}. 
\end{align*}
If $\kappa$ is not big enough, we can use multiplication by a Hasse invariant of parallel weight (such as the one of \cite{WushiMH}) to inject the forms of weight $\kappa$ into the ones of weight $\kappa'$, with $\kappa'$ big enough. Hence we have the first statement of the theorem.

Using the surjectivity of the reduction modulo $p$ for $p$-adic cusp forms given by Proposition \ref{prop:sesmodp}, one sees that the the statement of the theorem holds for ${\mr H^0(X^{\mr{tor}},\omega^{\kappa}(-D))}^{\mr{ord}} $ too.
\end{proof}
\begin{rmk}
Note that this is the only point where we need that the Hasse invariant has simple zeroes.
\end{rmk}
We  now define a space of $p$-adic modular forms we can control:
\begin{align*}
V^{\mr{cusp}} \colonequals \tilde{V}^{\mr{cusp}} \otimes_{\mc O\llbracket T_{\mr{Ig}} \rrbracket} \Lambda,
\end{align*}
for $\tilde{V}^{\mr{cusp}}$ the subset of cuspidal sections of $\tilde{V}$ and $T_{\mr{Ig}}$ the torus of the Galois group for Igusa tower.
\begin{prop}
The ordinary projector $e_{\mr{GL}}$ and $e_G$ are well-defined on $V^{\mr{cusp}}$.
\end{prop}
\begin{proof}
The only thing to check is that for every element $f$ of $V^{\mr{cusp}}$ the span of $\mb T^n_i f$ is of finite rank. But by definition $f$ is the reduction modulo $p^m$ of a classical forms of weight $\kappa$ and level $\G_1(p^l)$ for which this is clear. This show that $e_G$ converges and the same is true for $e_{\mr{GL}}$ (because it is contained in $e_G$).
\end{proof}
We say that a character $\kappa \in \mc W$ is algebraic if it comes via the Hodge--Tate map of Section~ \ref{sec:HodgeTate}  from an algebraic weight for classical modular forms as in Section~\ref{sec:classicalforms}.
Similarly we say that $\kappa$ is dominant (resp. regular, big enough, very regular) if the corresponding classical weight is dominant (resp. regular, big enough, very regular).

We can now apply Hida machinery and we get the main theorem.
\begin{teo}\label{thm:main}
We have constructed:
\begin{enumerate}
\item An  ordinary projector $e_G=e_G^2$  on $V^{\mr{cusp}}$ such that the Pontryagin dual of its ordinary part 
\begin{equation*}
V^{\mr{cusp},*,\mr{ord}}\colonequals \mr{Hom}_{\m Z_p}\left(V^{\mr{cusp},\mr{ord}},\m Q_p/\m Z_p\right)
\end{equation*}
(which is naturally a $\Lambda$-module) is finite free over $\Lambda^{\circ}$.
\item The $\Lambda$-module of Hida families $\mc S \colonequals \mr{Hom}_{\Lambda}\left(V^{\mr{cusp},*,\mr{ord}},\Lambda\right)$, which is of finite type over $\Lambda$.
\item  Given an algebraic weight $\kappa$, let $\mc P_{\kappa}$ be the corresponding prime ideal of $\Lambda$. Then
\begin{equation*}
   \mc S\otimes  \Lambda/\mc P_{\kappa}\stackrel{\sim}{\longrightarrow}\varprojlim_m\varinjlim_l V^{\mr{cusp}, \mr{ord}}_{m,l}[\tau],
\end{equation*}
and, if $\kappa$ is very regular, combining it with \eqref{eqn:classicaltopadic} gives
\begin{equation*}
   {\mr S_{\kappa}(\mc H, K)}^{\mr{ord}} \stackrel{\sim}{\longrightarrow} \left(\mc S \otimes\Lambda /\mc P_{\kappa}\right)[1/p].
\end{equation*}
Here the maps are equivariant under the action of the unramified Hecke algebra away from $Np$ and the $\m U_p$-operators.
\end{enumerate}
\end{teo}
\begin{proof}
We need to check that all the hypotheses of \cite{hida_control} are satisfied. (Hyp1) is verified by Proposition \ref{prop:sesmodp} and (Hyp2) is satisfied as the points $P_{\kappa}$ for $\kappa$ very regular are dense, as  seen in Proposition \ref{prop: properties of Lambda}; hence we can define the projectors $e_G$ and $e_{\mr{GL}}$. Moreover (C) holds by Lemma \ref{lemma:commute} and as already seen this implies that (F) holds  (see Theorem \ref{thm:boundedness}). 

We show that the last map is an isomorphism, {\it i.e.}\ the control theorem. 
We have the diagram 
 $$\xymatrixcolsep{3pc}\xymatrix{
  R[\kappa] \ar[r]\ar^{l_{\mr{can}}}[d] & F[\kappa]\ar[d]^{l_{\mr{can}}}\\
	{{W(k)}}[\kappa] \ar[r]& {{W(k)}}[\kappa] ,
}$$
where $R[\kappa]$ is the algebraic representation corresponding to $\omega^{\kappa}$, $F[\kappa]$ is the topological representation  obtained restricting  $R[\kappa]$ to $\mr{Gal}(\mr{Ig}/S^{\mr{tor}})$, and ${{W(k)}}[\kappa]$ is a one dimensional module where $T_{\mr{Ig}}$ acts via $\kappa$. Now, as already said \cite[Proposition 3.3]{PilHida}, whenever $\kappa$ is very regular the projector $e_{\mr{GL}}$ makes the map $l_{\mr{can}}$ into an isomorphism, and the same on the left (in which case it is enough for the weight to be dominant). 
This tells us that a function on $\mr{Ig}$ of weight $\kappa$ descends to a section of $\omega^{
\kappa}$ on $S^{\mr{tor}}_{\infty}$ whenever $\kappa$ is very regular.
As very regular implies big enough then congruence \ref{eq:TpUp}  holds and and we can write, as in Thereom \ref{thm:boundedness},
\begin{align*}
{\mr H^0(S^{\mr{tor}}_{\infty},\omega^{\kappa}(-D))}^{\mr{ord}} & \stackrel{\cong}{\rightarrow} { \left( \bigcup_{r} \frac{\mr H^0(X^{\mr{tor}},\omega^{\kappa'+r w(\mr{Ha}^{\mu})}(-D))}{{\mr{Ha}^{\mu}}^r } \right)}^{\mr{ord}} \\
& \stackrel{\cong}{\rightarrow} {\mr H^0(X^{\mr{tor}},\omega^{\kappa}(-D))}^{\mr{ord}}
\end{align*}
 and we are done.

We are left to show that $\mc S$ is free over $\Lambda^{\circ}$.
We proceed as in the proof of   \cite[Th\'eor\`eme 7.1]{PilHida}. 
Choose a point $\kappa$ and let $e_1,\ldots,e_l$ elements of $\mc S_{\chi} \colonequals \mc S \otimes_{\Lambda,\chi} \Lambda^{\circ}$ (for $\chi$ a finite order character of $T_{\mr{Ig}}$) which reduce to a basis of ${\mr S_{\kappa}(\mc H, \mc O)}^{\mr{ord}}$. 
We want to show that
$$\delta: \bigoplus \Lambda^{\circ} e_i \rightarrow \mc S_{\chi} $$
is an isomorphism. For $\kappa, \kappa'$ very regular, we have just seen that 
$${\mr S_{\kappa}(\mc H, \mc O_K)}^{\mr{ord}} \cong  {\mr S_{\kappa'}(\mc H, \mc O_K)}^{\mr{ord}}$$
and moreover by the proof of Theorem \ref{thm:boundedness}
$${\mr S_{\kappa}(\mc H, \mc O_K/p)}^{\mr{ord}} =  {\mr S_{\kappa'}(\mc H, \mc O_K/p)}^{\mr{ord}}.$$
Hence the kernel and cokernel of $\delta$ vanish modulo $P_{\kappa}$ for all such $\kappa$. By Zariski density, they must be $0$.
%the proof goes as in \cite[Proposition 3.6]{hida_control}.
\end{proof}

\begin{rmk}
A priori, the same results hold if one use $\mc O_K \llbracket T_{\rm{Ig}} \rrbracket$ and $\mc O_K \llbracket x_1,\ldots,x_{N} \rrbracket$ (for $N$ the $\Z_p$-rank of $T_{\rm{Ig}}$) in place of $\Lambda $ and $\Lambda^{\circ}$, as the points $P_{\kappa}$ of part (3) of the theorem are Zariski dense. Still, we find more interesting to have modules over $\Lambda$, as we feel that this ring is adapted to the classical weights. Moreover $\Lambda$ (or better the space $\mc W_{\mc LT_{f-1}}$) is used in $p$-adic Hodge theory for $\Z_{p^f}$, see \cite{BSX}.
\end{rmk}
%%%%%%%%%%%%%%%%%%%%%%%%%%%%%
\appendix
\section{\texorpdfstring{Reduced $\mu$-ordinary Hasse invariants}{Reduced mu-ordinary Hasse invariants}}\label{Appendix}
In \cite[Proposition 9.17]{Valentinhasse} Hernandez shows that in certain cases (namely, when the ordinary locus is not empty or when the $a_i$ are all distinct) his Hasse invariants are reduced. This is not always the case as in Remark 9.23 of {\it loc. cit.} he gives examples when they are not reduced. 
Let us recall the example: suppose that there is only one prime above $p$ in $F$ and suppose that $a_j=a_{j+1}$. Then we have \begin{align*}
 \mr{Fr}^{1}: M_{\sigma_j} \stackrel{\sim}{\longrightarrow}  M_{\sigma_{j+1}}, \\
 \mr{Fr}^{f-1}: M_{\sigma_{j+1}} \stackrel{\sim}{\longrightarrow}  M_{\sigma_j}.
\end{align*} If we denote by $\xi$ (resp. $\eta$) the determinant of $\bigwedge^{a_j}\mr{Fr}^{\alpha}/p^{c_j}$ (resp.  $\bigwedge^{a_j}\mr{Fr}^{f-\alpha}/p^{c_j}$) modulo $\mr{Fil}^1$, then 
\[
\mr{Ha}^{\mu,\sigma_j}= \xi \otimes \eta^{p}, \;\;\;\;  \mr{Ha}^{\mu,\sigma_{j+1}} =\eta \otimes \xi^{p^{f-1}}.
\]

The idea is to consider separately $\xi$ and $\eta$ (or more of such if more consecutive $a_j$'s are equal) to have partial Hasse invariants {\it \`a la} Goren \cite{Goren}.

We suppose for an instant that $p$ is inert in $F_0$ and let $\mc O$ be $F_p$ in case (U) and $F_{0,p}$ in case (L). We denote by $f$ the degree of $\mc O$. 
If we are in case (U), recall that we defined $a_{f/2 + i}=b_{f/2+1 -i}$.

For all $1\leq j \leq f$, let $\alpha \geq 0 $ be such that $a_j=\ldots=a_{j+\alpha}$, we define 
\begin{align*}
\mr{Ha}^{\mu,\sigma_{j+\beta}} \colonequals & \frac{\bigwedge^{d_j} \mr{Fr}}{p^{c_{\beta}}} : \frac{\bigwedge^{d_j}  M_{\sigma_{j+\beta}}}{\mr{Fil}^1 (\bigwedge^{d_j}  M_{\sigma_{j+\beta}})}
\rightarrow  \frac{\bigwedge^{d_j}   M_{\sigma_{j+\beta+1}}}{\mr{Fil}^1(\bigwedge^{d_j}   M_{\sigma_{j+\beta+1}} )} \mbox{ for }  0\leq \beta \leq \alpha-1;\\
\mr{Ha}^{\mu,\sigma_{j+\alpha}}\colonequals & \frac{\bigwedge^{d_j} \mr{Fr}^{f-\alpha}}{p^{c_{\beta}}} : \frac{\bigwedge^{d_j}  M_{\sigma_{j+\alpha}}}{\mr{Fil}^1 (\bigwedge^{d_j}  M_{\sigma_{j+\alpha}})}
\rightarrow  \frac{\bigwedge^{d_j}   M_{\sigma_{j}}}{\mr{Fil}^1(\bigwedge^{d_j}   M_{\sigma_{j}} )}
\end{align*}
for $d_j = \mr{dim} (\mr H^1_{\mr {dR}}(\mc A^{\mu})_{\sigma_{j}} / \mr{Fil}^1 (\mr H^1_{\mr {dR}}(\mc A^{\mu})_{\sigma_{j}})) = n - a_j$ and $c_j = \sum_{k=1}^{j-1} k (a_{k+1}-a_k)$.

Note that $M$ is isomorphic to its Cartier dual, hence we get $M \cong {\mc D (M)}^{\mr{Fr}^{f/2}}$ (meaning the underling $\mc O$-module is the same but the action of $\sigma$ on one side corresponds to the action of $\bar{\sigma}$ on the other). 

Because of this we get \begin{align*}
\frac{\bigwedge^{d_j}  M_{\sigma_{j+\beta}}}{\mr{Fil}^1 (\bigwedge^{d_j}  M_{\sigma_{j+\beta}})} & \cong \frac{\bigwedge^{d_j}  \mc D(M)^{\mr{Fr}^{f/2}}_{\sigma_{j+\beta}}}{\mr{Fil}^1 (\bigwedge^{d_j}  M_{\sigma_{j+\beta}})} \\
& \cong \frac{\bigwedge^{d_j}  \mc D(M)_{\sigma_{f -(j+\beta)+1}}}{\mr{Fil}^1 (\bigwedge^{d_j}  M_{\sigma_{f -(j+\beta)+1}})} 
\end{align*} 
and also $$ \mr{Fil}^1 M_{\sigma_{i}} \cong M_{\sigma_{f-i+1}}/\mr{Fil}^1 M_{\sigma_{f-i+1}} $$ 
This implies $\mr{Ha}^{\mu,\sigma_{j+\beta}} = \mr{Ha}^{\mu,\bar{\sigma}_{j+\beta}} $.

In case (L) instead we split $M=M^+\oplus M^-$ and we repeat the same definitions with $M^+$ in place of $M$, {\it i.e.} we consider only the Dieudonn\'e module of $\mc A^{\mu}[(\pi^+)^{\infty}]$.

\begin{defi}
For all primes $\pi_i$ we define $\mr{Ha}_{\pi_i}^{\mu}= \prod_{i=1}^{d}\mr{Ha}^{\mu,\sigma_{i}}$ and in general $$\mr{Ha}^{\mu}\colonequals \prod_{i=1}^k \mr{Ha}^{\mu}_{\pi_i}.$$
\end{defi}

\begin{teo}\label{teo:appendix}
The divisor $V(\mr{Ha}^{\mu})$ in $S_1^{*}$ is reduced.
\end{teo}
\begin{proof}
It is enough to show that the same holds for $S_1$. Moreover, as everything splits according to the action of $\mc O_F \otimes \m Z_p$, it is not harmful to work with a BT-with $\mc O$-structure, for $\mc O$ any of $\mc O_{F,\pi_i}$ (if $\pi_i$ is inert in $F/F_0$) or $\mc O_{F_0,\pi_i}$ (if $\pi_i$ splits).

As in Proposition \ref{prop:zerooutside} we consider a Newton stratum $\mc N_j$ of abelian varieties whose Newton polygon $\mr{NP}_j$ coincides with the $\mu$-ordinary Newton polygon $\mr{NP}$ except at the breaking point $a_j$ (and the dual point $b_j$).  First of all notice that if $a_i \neq a_j$ then  $\mr{Ha}^{\mu,\sigma_{i}}$ doesn't distinguish between $\mr{NP}_j$ and $\mr{NP}$. So we have to analyse only $\mr{Ha}^{\mu,\sigma_{i}}$ for $j \leq i \leq j+\alpha$ with $\alpha \geq 0$ and $a_j=\ldots=a_{j+\alpha}$.
Only for this theorem, we let $H\colonequals\mc A_{j+\alpha+1}/\mc A_{j-1}$, for $\set{\mc A_k}$ the filtration for the $p$-torsion of the universal abelian variety over $\mc N_j$. It is a BT with $\mc O$-structure of $\mc O$-height $2$. We follow very closely the proof of \cite[Proposition 9.17]{Valentinhasse}: we need to study the matrices of the Frobenius on the Dieudonn\'e module $M_H$ of $H$. We let $\set{ e_{1,i},e_{2,i}}_{i=1,\ldots,f}$ be a basis of $M_{H}$ such that $\mc O$ acts via $\sigma_i$ on $\set{ e_{1,i},e_{2,i}}$. We have  
\begin{align*}
\mr{rk}_{\mc O} \mr{Fil}^{1}M_{H,\sigma_i} = \left\{ \begin{array}{cc}
0 & \mbox{ if } i < j\\
1 & \mbox{ if } j \leq i \leq j+\alpha\\
2 & \mbox{ if } i > j+\alpha
\end{array} \right. .
\end{align*}

By our choice of $\mr{NP}_j$ we know that the   two slopes of $\mr{Fr}^f$ on $\set{ e_{1,i},e_{2,i}}$ (seen as a $\mr{Fr}^f$ crystal of rank two) are the same and that the dimension of $H$ is $\alpha+1$. The Dieudonn\'e--Manin classification of Dieudonn\'e modules tells us that $H$ is isogenous over ${\m Q}^{\mr{ur}}_p $ to 
\begin{align*}
M_{2,s} = {\m Q}^{\mr{ur}}_p[\Phi]/(\Phi^2 - p^{s}),
\end{align*}
if $s=2(f-j-\alpha)+\alpha+1$ is odd, and 
\begin{align*}
M_{1,s/2} \oplus M_{1,s/2} = {\m Q}^{\mr{ur}}_p[\Phi]/(\Phi - p^{s/2})\oplus {\m Q}^{\mr{ur}}_p[\Phi]/(\Phi - p^{s/2}) ,
\end{align*}
 if $s$ is even. Here $\Phi$ is $\mr{Fr}^f$.

%The $2f  \times 2f$ matrix of $F$ on $M_H$ is given by 
%\begin{align*}
%\left(  
%\begin{array}{cccccccc}
%0                        &  0                     &    \cdots              &      &          &              &   0   &p \mr{Id}_2 \\
%\mr{Id}_2           & 0                     & \cdots    &      &                       &                 &   0  &0\\ 
%0                         &  \mr{Id}_2     &               &       &                        &                  &     &\vdots \\
%\vdots                 & 0                   & \ddots   & 0      &                        &                  &      &\\
%                          &                        &        0     & A     &          0             &                  &       &\\
  %                         &                        &              &  0      &  p\mr{Id}_2 &          0          &      & \\
 %                          &                        &              &        &        0            & \ddots         &    0     &    \\      
%                            &                        &              &        &                     &                   &  p\mr{Id}_2    &    0         
%\end{array}
%\right) 
%\end{align*}
%where each entry is a $2 \times 2$ matrix   except $A$ which has size $2(\alpha+1)\times 2(\alpha+1)$. 

We are interested in the matrix of the isogeny $\mr{Fr}: M_{\sigma_j} \rightarrow M_{\sigma_{j+1}}$. It can only be $\left(  
\begin{array}{cc}
0 & p  \\
1 & 0
\end{array}
\right)$ otherwise the slopes would be distinct.

We know that the deformation space of $H$ has dimension $\alpha+1$ \cite[\S 2.1.4]{Moonen} and thanks to \cite[Proposition 2.1.8]{Moonen} the matrix of $\mr{Fr}$ on the the universal BT group with $\mc O$-structure of type as $H$ is 
\begin{align*}
\left(  
\begin{array}{cc}
X_1 & p  \\
1     & 0
\end{array}
\right) 
\end{align*}
and, similarly, for all $ \beta \leq \alpha-1 $ the universal matrix for 
$\mr{Fr}: M_{\sigma_j+\beta} \rightarrow M_{\sigma_{j+\beta+1}}$ is 
\begin{align*}
\left(  
\begin{array}{cc}
X_{\beta+1}& p \\
1 & 0
\end{array}
\right).
\end{align*}
In particular if $\beta \neq \alpha$ we have 

\begin{gather*}
\mr{Fr} \colon  M_{H,\sigma_{j+\beta}}/\mr{Fil}^{1}(M_{H,\sigma_{j+\beta}}) \to  M_{H,\sigma_{j+\beta+1}}/\mr{Fil}^{1}(M_{H,\sigma_{j+\beta+1}}) \\
 e_{1,j+\beta} \mapsto (X_{\beta+1}e_{1,j+\beta+1} + e_{2,j+\beta+1})/e_{2,j+\beta+1} = X_{\beta+1}e_{1,j+\beta+1}
\end{gather*}
So these $\mr{Ha}^{\mu,\sigma_{j+\beta}}$ are non-zero on the tangent space at a point in $\mc N_j$, and their images are all independent. For $\mr{Ha}^{\mu,\sigma_{j+\alpha}}$ a similar calculation (here we have to take a power of $\mr{Fr}$ and divide by a power of $p$, and note that the remaining $M_{H,\sigma_i}$ do not contribute to the tangent space as they are ``rigid'') gives that on the tangent space $\mr{Ha}^{\mu,\sigma_{j+\alpha}}$ is $X_{\alpha+1}$. In particular the product is a non-zero linear form on the tangent space. 
\end{proof}

\begin{rmk}
Note that Lemma \ref{lemma:commute} is not affected if one uses this new Hasse invariant as the quotient of the lattice of $\mc A$ by the one of $\mc A /L_i$ is free over $\mc O/p \otimes_{\m F_p}\mc O/p$.
\end{rmk}

\begin{rmk}The Hasse invariant just defined is a product of  modular forms of non-parallel weight $(0,\ldots,-1,p,0,\dots)$ or $ (0,\ldots,p^{f-\alpha}, 0,\ldots, 0, -1, 0,\ldots)$ and it is non-cuspidal modulo $p$ so in general one can not expect to lift (not even a power of it) to characteristic zero.
\end{rmk}

\bibliographystyle{amsalpha}
\bibliography{biblio}
\end{document}